\documentclass[12pt, reqno, 
]{amsart}
\usepackage{amsmath}
\usepackage{amstext}
\usepackage{amsbsy}
\usepackage{amsopn}
\usepackage{upref}
\usepackage{amsthm}
\usepackage{amsfonts}
\usepackage{amssymb}
\usepackage{mathrsfs}
\usepackage{times}
\usepackage{fancybox} 
\usepackage{ascmac}
\allowdisplaybreaks
 \newtheorem{theorem}{Theorem}[section]
 
 \newtheorem{proposition}[theorem]{Proposition}
 
\theoremstyle{definition}
 \newtheorem{definition}[theorem]{Definition}
 \theoremstyle{remark} 
 \newtheorem{remark}[theorem]{Remark}

\newcommand{\p}{\partial}
\newcommand{\RR}{\mathbb{R}}
\newcommand{\TT}{\mathbb{T}}
\newcommand{\UU}{\mathcal{U}}
\newcommand{\VV}{\mathcal{V}}
\newcommand{\WW}{\mathcal{W}}
\newcommand{\TW}{\widetilde{\mathcal{W}}}
\newcommand{\lr}[1]{\left\langle #1 \right\rangle}
\numberwithin{equation}{section}
\begin{document}
\title[Uniqueness of 1D Generalized Bi-Schr\"odinger Flow]
{Uniqueness of 1D Generalized Bi-Schr\"odinger Flow
}
\author[E.~Onodera]{Eiji Onodera}
\address[Eiji Onodera]{Department of Mathematics, 
Faculty of Science and Technology, 
Kochi University, 
Kochi 780-8520, 
Japan}
\email{onodera@kochi-u.ac.jp}
%
\maketitle
\begin{abstract}
We establish the uniqueness of a smooth generalized 
bi-Schr\"odinger flow from the one-dimensional flat torus 
into a compact locally Hermitian symmetric space. 
The governing equation, 
which is satisfied by sections of the pull-back bundle 
induced from the flow, 
is a fourth-order nonlinear dispersive partial differential equation 
with loss of derivatives. 
To show the uniqueness, we adopt an extrinsic approach 
to compare two solutions via an isometric embedding into 
an ambient Euclidean space.  
We introduce an energy modifying the classical $H^2$-energy 
for the difference of two solutions, 
the detailed estimate of which enables us to eliminate 
the difficulty of the loss of derivatives. 
In particular, we demonstrate how to decide the form of 
the modification by exploiting the geometric structure 
of the locally Hermitian symmetric space.  
\end{abstract}
\section{Introduction}
\label{section:introduction}
The so-called generalized bi-Schr\"odinger flow equation 
for maps from a Riemannian manifold into K\"ahler 
manifold was recently introduced  
by Ding and Wang in \cite{DW2018}
in the following way:   
Let $(M, g)$ be an $m$-dimensional Riemannian manifold 
with a metric $g$
and let $(N,J,h)$ be a $2n$-dimensional K\"ahler manifold 
with the complex structure $J$ and a K\"ahler metric $h$.
Let $\alpha, \beta, \gamma\in \RR$ be constants where $\beta\ne 0$. 
The energy functional $E_{\alpha,\beta,\gamma}(u)$ 
for smooth maps
$u:(M, g)\to (N, J, h)$ is defined by  
$$
E_{\alpha,\beta,\gamma}(u)
:=
\alpha\,E(u)
+\beta\,E_2(u)
+
\gamma\,E_{\star}(u).  
$$
Here,
$E(u)
=
\frac{1}{2}
\int_M
|\nabla u|^2\,dv_g$
and 
$E_2(u)
=
\frac{1}{2}
\int_M
|\tau(u)|^2\,dv_g$ 
are energy functionals whose 
critical points are respectively known as 
harmonic maps and bi-harmonic maps, 
where  
$\nabla u$ is a section of the vector-bundle 
$T^{\ast}M\otimes u^{-1}TN$, 
$\tau(u)$ is the tension field of $u$, 
and $dv_g$ is the 
volume form of $(M,g)$. 
The energy functional $E_{\star}(u)$ is defined by 
\begin{align}
E_{\star}(u)
&:=
\int_M
h(R^N(\nabla u, J_u\nabla u)J_u\nabla u, \nabla u)\,dv_g, 
\nonumber
\end{align} 
where 
$R^N(\cdot,\cdot)$ is the Riemann curvature tensor on $(N,J,h)$. 
A time-dependent map 
$u=u(t,x):(-T,T)\times M\to N$  
is called a generalized bi-Schr\"odinger flow 
from $(M,g)$ to $(N,J,h)$ 
if $u$ satisfies the following Hamiltonian gradient flow equation 
\begin{equation}
u_t=J_u\nabla E_{\alpha,\beta,\gamma}(u)
\label{eq:bibi} 
\end{equation}
on $(-T,T)\times M$ for some $T>0$. 
We can see \eqref{eq:bibi} as a 
fourth-order extension of the well-known 
Schr\"odinger map flow equation -- a second-order 
geometric dispersive partial differential equation(PDE) -- formulated by 
\begin{equation}
u_t=J_u\tau (u), 
\label{eq:SM}
\end{equation}
which is just \eqref{eq:bibi} under the assumption 
$(\alpha,\beta,\gamma)=(-1,0.0)$.
For more details, see \cite{DW2018}. 
\par
Another unified formulation of \eqref{eq:bibi} 
was derived by the present author in \cite{onodera4}
for time-dependent maps from the real line $\RR$ 
or the one-dimensional flat torus $\TT:=\RR/ 2\pi \mathbb{Z}$ 
into a locally Hermitian symmetric space. 
Let $(N,J,h)$ be a locally Hermitian symmetric space,  
which is a complex manifold characterized by the condition 
$\nabla^N J=\nabla^N R^N=0$.  
Here $\nabla^N$ denotes the Levi-Civita condition on $(N,J,h)$. 
Throughout this paper, 
we adopt the definition of $R^N$ by 
$R^N(X,Y)Z:=\nabla^N_X\nabla^N_YZ-\nabla^N_Y\nabla^N_XZ
-\nabla^N_{[X,Y]}Z$ 
for any $X,Y,Z\in \Gamma(TN)$.  
Then, the result of \cite{onodera4} shows 
\eqref{eq:bibi} 
for time-dependent maps $u$ from $\RR$ or $\TT$ 
into $(N,J,h)$ can be written by 
\begin{align}
u_t&=
\beta\,J_u\nabla_x^3u_x
-\alpha\, J_u\nabla_xu_x
\nonumber
\\&\quad
+(\beta+8\gamma)\, R^N(\nabla_xu_x,u_x)J_uu_x
-12\gamma\, R^N(J_uu_x,u_x)\nabla_xu_x.
\label{eq:4101}
\end{align} 
Here, 
$u_t=du(\frac{\p}{\p t})$,  
$u_x=du(\frac{\p}{\p x})$, 
$du$ denotes the differential of $u$,
$\nabla_x$ denotes the covariant derivative along $u$ 
with respect to $x$, 
$J_u$ denotes the complex structure at $u=u(t,x)\in N$.
\par 
Restricting ourselves to the case $(N,J,h)$ 
is the canonical $2$-sphere 
$\mathbb{S}^2$ with $\alpha=-1$, 
we find \eqref{eq:4101} arises 
in mathematical physics.
Indeed, the $\mathbb{S}^2$-valued model 
\eqref{eq:4101} in this context is 
derived by a geometric reformulation of a
continuum limit
of the Heisenberg spin chain systems 
with nearest neighbor bilinear and
bi-quadratic exchange 
interactions(\cite{LPD}) .
One can consult with \cite[Section~2.2]{onodera4} for the 
reformulation of the physical model as \eqref{eq:4101}.
Moreover,
the $\mathbb{S}^2$-valued model \eqref{eq:4101} 
also occurs in relation 
with the so-called Fukumoto-Moffatt model 
equation(\cite{fukumoto, FM}) describing the 
motion of a vortex filament in an incompressible 
prefect fluid in $\RR^3$. 
In addition, by \cite{fukumoto, FM,LPD}, 
the $\mathbb{S}^2$-valued model is known to be completely 
integrable under the additional assumption $\beta=-8\gamma$, 
in that it has infinitely number of conservation laws.
\par 
In this paper, 
letting $(N,J,h)$ be a compact locally 
Hermitian symmetric space 
and restricting the spatial domain to $\TT$, 
we consider the following initial value problem 
\begin{alignat}{2}
 & u_t
  =
  a\,J_u\nabla_x^3u_x
  +
  \lambda\, J_u\nabla_xu_x
  +
  b\, R^N(\nabla_xu_x,u_x)J_uu_x
  +
  c\, R^N(J_uu_x,u_x)\nabla_xu_x,
\label{eq:pde}
\\
& u(0,x)
  =
  u_0(x),
\label{eq:data}
\end{alignat}
where a map
$u=u(t,x):\RR\times \TT\to N$  
is the solution being a flow of closed curve on $N$
and   
$u_0=u_0(x):\TT\to N$ is the given initial closed curve on $N$. 
Moreover, 
$a\ne 0$, $b$, $c$, $\lambda$ are supposed to be 
real constants so that \eqref{eq:pde} is handled as 
a fourth-order nonlinear dispersive PDE. 
Obviously, the equation \eqref{eq:pde} with
$(a,\lambda,b,c)=(\beta,-\alpha,\beta+8\gamma,-12\gamma)$
is nothing but \eqref{eq:4101}. 
In other words,  
\eqref{eq:pde} slightly extends 
\eqref{eq:4101} in the sense 
the relation $c=3(a-b)/2$
is not imposed among the constants $a,b,c$. 
\subsection{Main results and related known results}
The purpose of our research was to solve 
\eqref{eq:pde}-\eqref{eq:data}.
This is a fundamental problem in the theory of PDEs. 
Moreover, making the relationship between 
the solvablity and the geometric setting of $(N,J,h)$ clear 
is fascinating also from the viewpoint of geometric analysis. 
In this part, we state the known results in this direction 
and our main results in this paper. 
\par
The local and global well-posedness for 
the Schr\"odinger map flow equation \eqref{eq:SM}
with data in Sobolev spaces has been studied extensively 
by many authors. 
We do not attempt to survey all the results, but refer to 
\cite{CSU, chihara, DW1998, DW0, koiso, McGahagan, NSVZ, 
PWW, RRS, SSB, TU} for more details. 
We note that the local existence and uniqueness results 
can be obtained under the 
K\"ahler condition $\nabla^N J=0$ on $(N,J,h)$ by 
exploiting the classical geometric energy method based on the 
integration by parts and the Sobolev embedding.
If $\nabla^N J\ne 0$, then the so-called a loss of derivative occurs. 
In other words, the classical energy method breaks down 
due to the presence of a bad term coming from $\nabla^NJ\ne 0$.    
Fortunately however, the difficulty was overcame in \cite{chihara} 
by developing the energy method with a gauge transformation 
acting on sections of the pull-back bundle $u^{-1}TN$.
In addition, some third-order generalizations of  
\eqref{eq:SM} have also been investigated in
\cite{CO, onodera0, onodera1,onoderag, Song, 
SW2010, SW2011, SW2013}.
We do not attempt to state the detail, but 
stress that the classical geometric energy method still works 
to obtain local existence and uniqueness results
under the K\"ahler condition $\nabla^N J=0$. 
\par 
In contrast, for our equation \eqref{eq:pde}, 
we find the difficulty of loss of derivatives occurs 
even if $\nabla^N J=0$ holds.  
In general, 
the solvability of the initial value problem for 
a dispersive PDE essentially depends on the 
structure of the derivatives of the solution in the equation
(see, e.g, \cite{Akhunov, chihara2, Mizohata, Mizuhara, Tarama}). 
Thus the crucial part of our problem is to reveal the relationship 
between the structure and the setting of $(N,J,h)$. 
The procedure becomes harder as the spatial order of 
the equation becomes higher, and the fourth-order case is the
first one we encounter the difficulty of loss of derivatives 
under $\nabla^NJ=0$.
Furthermore, as the spatial domain $\TT$ is compact, 
the so-called dispersive smoothing effect inherited to the solution  
on $\RR$ -- which was useful to compensate 
the loss of derivatives completely as in \cite{CO2} -- 
is absent in our setting. 
Therefore, 
a stronger geometric structure of $(N,J,h)$
is required even to establish local existence results. 
\par 
Having the above in mind, we state 
the known results on \eqref{eq:pde} in this direction.  
Guo, Zeng, and Su in \cite{GZS} showed the local existence 
of weak solutions to \eqref{eq:pde}-\eqref{eq:data} 
when $N=\mathbb{S}^2$ with $\lambda=1$, $a\ne 0$,
$c=3(a-b)/2$, and $b=0$.  
Note that the equation \eqref{eq:pde} in the setting without $b=0$ 
corresponds to the 
$\mathbb{S}^2$-valued physical model of continuum limit of the 
Heisenberg spin chain systems(\cite{LPD}) as stated above. 
The additional assumption $b=0$ corresponds to the 
completely integrable condition
and the proof in \cite{LPD} is essentially based on a conservation law 
which is absent unless $b=0$ is imposed.
After that, 
the present author \cite{onodera2} showed the local 
existence and uniqueness of a 
smooth solution to \eqref{eq:pde}-\eqref{eq:data} when  
$N=\mathbb{S}^2$ with $\lambda=1$, $a\ne 0$ and   
$c=3(a-b)/2$ without any assumptions on the constants.  
In \cite{onodera3}, 
he furthermore extended the results to the case $N$ is a compact 
Riemann surface with constant Gaussian curvature,  
without any additional assumptions on the constants 
except for $a\ne0$.  
\begin{remark}
\label{remark:bor}
In more detail, the equation handled in \cite{onodera3} 
stated above is formulated by 
\begin{equation}
\label{eq:510add}
u_t=b_1J_u\nabla_x^3u_x
+b_2J_u\nabla_xu_x
+b_3h(u_x,u_x)J_u\nabla_xu_x
+b_4h(\nabla_xu_x,u_x)J_uu_x
\end{equation}
where $b_1(\ne 0),b_2,b_3,b_4\in \RR$ are constants.
This is different from \eqref{eq:pde}  
unless $(N,J,h)$ is a Riemann surface with constant Gaussian curvature. 
Moreover, as far as \eqref{eq:510add} is considered,   
it is unlikely that we can extend the local existence result to 
$(N,J,h)$ wider than the class 
of compact Riemann surfaces with constant curvature. 
This was first pointed out by Chihara in \cite{chihara2} from the 
theory of $L^2$-well-posedness for linear dispersive PDE systems, 
where the case of the Riemann surface as $(N,J,h)$ was discussed. 
The present author also investigated \eqref{eq:510add}
in case of higher-dimensional manifolds, and found that 
local existence result still holds as far as 
$(N,J,h)$ is a compact $2n$-dimensional K\"ahler manifold 
with constant non-vanishing sectional curvature 
(in the sense of a real manifold). 
However, it is meaningless, because the class of such 
manifolds is known to be the empty set without $n=1$ 
(see, e.g., \cite[Theorem~7.1.]{Hsiung}). 
Note again the strong obstruction is true only 
of \eqref{eq:510add}, 
and not of \eqref{eq:pde}. 
\end{remark}
Let us go back to our problem \eqref{eq:pde}-\eqref{eq:data}. 
Recently, the present author in \cite{onodera4} showed 
only the local existence of the solution to 
\eqref{eq:pde}-\eqref{eq:data}
for time-dependent maps into a 
compact locally Hermitian symmetric space.     
Indeed, by the mix of the parabolic regularization 
and the geometric energy method combined with 
a gauge transformation acting on $\Gamma(u^{-1}TN)$, 
he showed the following:
\begin{theorem}[\cite{onodera4},Theorem~1.3.]
\label{theorem:existence}
Suppose that $(N,J,h)$  is a compact 
locally Hermitian symmetric space. 
Let $k$ be an integer satisfying $k\geqslant 4$. 
Then for any 
$u_0\in C(\mathbb{T};N)$ satisfying 
$u_{0x}\in H^k(\TT;TN)$, 
there exists $T=T(\|u_{0x}\|_{H^4(\TT;TN)})>0$
such that 
\eqref{eq:pde}-\eqref{eq:data} 
has a solution 
$u\in C([-T,T]\times \TT;N)$
satisfying 
$u_x\in 
L^{\infty}(-T,T;H^{k}(\TT;TN))
\cap
C([-T,T];H^{k-1}(\TT;TN)).
$
\end{theorem}
Here and hereafter, $\Gamma(u^{-1}TN)$ denotes the set of all 
sections of the pull-back bundle $u^{-1}TN$, 
and $H^k(\TT;TN)$ 
is defined to be the set of all  sections
$V\in \Gamma(u^{-1}TN)$ 
such that  
$$
\|V\|_{H^k(\TT;TN)}
:=
\sum_{\ell=0}^k
\int_{\TT}
h(\nabla_x^{\ell}V(x), \nabla_x^{\ell}V(x))\,dx
<\infty.
$$
\par
It is natural to investigate whether the uniqueness of the solution 
to \eqref{eq:pde}-\eqref{eq:data} holds or not.  
The question is rather challenging, 
since how to apply the condition $\nabla^NR^N=0$
is unclear, unlikely to the proof of existence.  
Indeed, there are no other uniqueness results except  for
the very limited case of $(N,J,h)$  
as stated above.
The purpose of this paper is to establish how to apply 
$\nabla^NR^N=0$ and to prove the uniqueness.  
Our main result is now stated as follows:  
\begin{theorem}
\label{theorem:uniqueness}
Suppose that $(N,J,h)$ is a compact locally Hermitian symmetric 
space.
Let $k$ be an integer satisfying $k\geqslant 5$. 
Let $u$ and $v$ be solutions to \eqref{eq:pde}-\eqref{eq:data} 
in Theorem~\ref{theorem:existence}.
Then it follows that $u=v$ on $[-T,T]\times \TT$.
\end{theorem}
\begin{remark}
Let $u$ be a solution to \eqref{eq:pde}-\eqref{eq:data} 
with regularity in Theorem~\ref{theorem:existence}. 
Once the uniqueness is established, then we can recover 
the time-continuity of the highest derivative 
$\nabla_x^ku_x$ in $L^2$, 
by applying the weak time-continuity in $L^2$ 
and the energy estimate 
for $\|u_x(t)\|_{H^k(\TT;TN)}^2$ established in \cite{onodera4}.
This shows $u_x\in C([-T,T];H^{k}(\TT;TN))$.  
The argument is now standard, and hence we omit the detail.  
\end{remark} 
\begin{remark}
Our proof of Theorem~\ref{theorem:uniqueness} 
works without modifications even if we replace $\TT$ with $\RR$. 
Nonetheless, no originality is claimed here, 
because the method previously established 
in \cite{CO2} based on the dispersive smoothing effect 
works to prove in the case of $\RR$.
\end{remark}
\begin{remark}
The class of compact locally Hermitian symmetric spaces as $(N,J,h)$ 
includes all compact K\"ahler manifolds of 
constant holomorphic sectional curvature 
and compact Hermitian symmetric spaces, 
as well as compact Riemann surfaces with constant curvature.
We should mention that   
the authors in \cite{DW2018} obtained the explicit formulation 
of \eqref{eq:bibi} for time-dependent maps from $\RR$ or $\TT$ into 
three types of Hermitian symmetric spaces as $(N,J,h)$ 
by using the Lie bracket in the symmetric Lie 
algebra of $N$. 
In particular, they obtained the explicit formulation 
when $N$ is a compact K\"ahler Grassmannian manifold
$G_{n,k}$ for $n,k$ with 
$1\leqslant k\leqslant n-1$. 
Though the formulation of the equation is seemingly different from \eqref{eq:pde}, 
the uniqueness result for the equation also falls within the scope of Theorem~\ref{theorem:uniqueness}.
\end{remark}
\begin{remark}
In \cite{DW2018}, 
the equivalence of 
\eqref{eq:bibi}  
for time-dependent maps from 
$\RR$ into $N=G_{n,k}$ 
and a fourth-order nonlinear dispersive PDE for 
$k\times (n-k)$-complex-matrix-valued functions is also discussed. 
Particularly in the case $N=G_{2,1}$, 
corresponding equation for matrix-valued functions 
is just a (single) fourth-order semilinear dispersive PDE 
for complex-valued functions, 
and the well-posedness of the initial value problem in a Sobolev space 
was established by Segata in \cite{segata}, 
including the case the spatial domain is $\TT$. 
In the higher-dimensional case of $N$ 
except for $N=G_{2,1}$, 
we can see the corresponding equation for matrix-valued functions 
as a system of nonlinear dispersive PDEs 
including a nonlocal nonlinearity, 
which is more attractive in the interface of 
geometry and analysis of nonlinear dispersive PDEs. 
To the best of the author's knowledge, however, 
there are no results on the solvability of their initial value problem. 
If the equivalence holds also on the spatial domain $\TT$, 
then Theorems~\ref{theorem:existence} and \ref{theorem:uniqueness}
automatically give the local existence and uniqueness results 
on the dispersive PDEs for matrix-valued functions.   
Although it seems to be a nontrivial matter to show the equivalence,  
we are strongly convinced that the insights obtained 
in this paper give rise to valuable information about 
how to handle the dispersive PDEs for matrix-valued functions. 
\end{remark}
\subsection{Key of the proof}
In this part, we state the key idea of the proof of Theorem~\ref{theorem:uniqueness} after reviewing the proof 
of Theorem~\ref{theorem:existence} briefly. 
\par 
Theorem~\ref{theorem:existence} on the local existence 
of a solution was proved by an intrinsic approach in \cite{onodera4}. 
To state the key observation,  
suppose $u$ is a smooth solution 
to \eqref{eq:pde}-\eqref{eq:data}. 
Then 
the equation satisfied by 
$\nabla_x^ku_x$ with $k\geqslant 4$ 
turns out to be described by  
\begin{align}
(\nabla_t-a\,J_u\nabla_x^4-d_1\,P_1\nabla_x^2-d_2\,P_2\nabla_x)
\nabla_x^ku_x
&=
O\left(
\sum_{m=0}^{k+2}
|\nabla_x^mu_x|_h
\right), 
\label{eq:esspde}
\end{align}
where $|\cdot|_h=\left\{h(\cdot,\cdot)\right\}^{1/2}$, 
and $d_1$ and $d_2$ are real constants depending on 
$a,b,c,k$, 
and
\begin{align}
P_1Y
&:=
R^N(Y,J_uu_x)u_x 
\quad
\text{and}
\quad
P_2Y
:=
R^N(J_u\nabla_xu_x,u_x)Y
\nonumber
\end{align}
respectively 
for any $Y\in \Gamma(u^{-1}TN)$. 
From \eqref{eq:esspde},  
we find the classical energy estimate 
for $\|u_x\|_{H^k(\TT;TN)}^2$ 
breaks down because 
$d_1\,P_1\nabla_x^2$ and $d_2\,P_2\nabla_x$
cause loss of derivatives.  
Fortunately however, the difficulty 
was overcame in \cite{onodera4} by 
the geometric energy method combined with 
a gauge transformation acting on $\Gamma(u^{-1}TN)$. 
Observing the method in \cite{onodera4} in more detail, 
we find that the reason we can construct the gauge transformation  
comes from the following good properties:
\begin{itemize}
\item[(A)] $(J_uP_1+P_1J_u)J_u$ is skew-symmetric on 
$\Gamma(u^{-1}TN)$. 
\item[(B)]  
$P_2$ has a potential in the sense
$P_2=(\nabla_x \widetilde{P})$ where 
$\widetilde{P}=\frac{1}{2}R^N(J_uu_x,u_x)$.
\end{itemize}
We call them a
``good structure'' 
of \eqref{eq:pde} in this paper.
We note also that
the right hand side of \eqref{eq:esspde} 
also includes $\nabla_x^2(\nabla_x^ku_x)$ 
and $\nabla_x(\nabla_x^ku_x)$, 
but loss of derivatives do not occur from the part 
thanks to the assumption $\nabla ^NR^N=\nabla^N J=0$.  
It is unlikely that we can relax the assumption $\nabla^N R^N=0$, 
since the assumption  
seems to correspond to the constant curvature condition on the 
equation in \cite{onodera3}. 
Indeed, if we let $\nabla^N R^N\ne 0$, 
then  \eqref{eq:esspde} involves a skew-symmetric 
first-order derivative of the form 
$(\nabla_x R^N)(J_uu_x,u_x)\nabla_x(\nabla_x^ku_x)$, 
which is the worst term we cannot handle in the energy estimate. 
\par
Now we turn our attention to the proof of 
Theorem~\ref{theorem:uniqueness}. 
To state the key of the proof simply, suppose 
$u$ and $v$ are sufficiently smooth solutions to 
\eqref{eq:pde}-\eqref{eq:data} with same initial data. 
It suffices to show $u=v$. 
Since their difference is not defined on $N$ directly,   
we fix an isometric embedding $w$ from $(N,J,h)$ 
into some ambient Euclidean  space $\RR^d$ with sufficiently 
large integer $d$, and define  
$$
Z:=U-V,\quad
U:=w{\circ}u,\quad 
V:=w{\circ} v,
$$
as vector-valued functions with values in $\RR^d$. 
However, then, the extrinsic formulation for the equation 
satisfied by $U$
becomes highly nonlinear (quasilinear) fourth-order dispersive equation   
involving $\p_x^3U$ as well as $\p_x^2U,\p_xU$ in the nonlinearity   
and the good structure such as \eqref{eq:esspde} 
is lost from the equation satisfied by 
$Z$ and the derivatives in $x$.
Nonetheless, we expect that the good structure still remain 
at least in the tangential component.  
Having them in mind, 
we decompose the equation satisfied by the second derivative of $Z$ in $x$ 
into the tangential component in $dw(\Gamma(u^{-1}TN))$  
and normal component in $(dw(\Gamma(u^{-1}TN))^{\perp}$, 
and exploit the good structure hidden  
in $dw(\Gamma(u^{-1}TN))$
to derive the estimate.  
The normal component is estimated by making use of 
some properties of the second fundamental form on $N$. 
More precisely, setting 
$$
\UU:=dw_u(\nabla_xu_x),\quad  
\VV:=dw_v(\nabla_xv_x), \quad 
\WW:=\UU-\VV, 
$$
we compute the equation satisfied  by $Z$, $Z_x$ and $\WW$ 
(not by $\p_x^2Z$ in order to make 
the computation a little simpler). 
Particularly, 
the results of the computation for $\WW$ is as follows: 
$$
\WW_t=a\p_x^2\{J(U)\p_x^2\WW\}
+\mathcal{N}(\p_x\WW,\p_x^2\WW,\p_x^3\WW)
+\mathcal{O}(|Z|+|Z_x|+|\WW|), 
$$ 
where for each $(t,x)$,
$J(U(t,x)):\RR^d\to \RR^d$ is a map behaving as 
an almost complex structure on $dw_{u(t,x)}(T_{u(t,x)}N)$,  
and the nonlinearity $\mathcal{N}$ is an 
$\RR^d$-valued function involving 
$\p_x^3\WW$, $\p_x^2\WW$, $\p_x\WW$.  
The explicit expression will be given by \eqref{eq:WW}  
(see also Proposition~\ref{proposition:UU}), 
and all the geometric notion used to describe \eqref{eq:WW} 
will be defined in Section~\ref{section:GP}. 
Looking at \eqref{eq:WW}, \eqref{eq:UU2}, 
and the classical energy estimate 
\eqref{eq:momo} coming from them, 
we notice that the loss of derivatives 
in the classical energy estimate for $\WW$ 
comes only from the linear combination of 
$$
R(\p_x^2\WW,J(U)U_x)U_x 
\quad \text{and} \quad
R(J(U)\UU,U_x)\p_x\WW,
$$ 
where the definition of $R$ will be given 
in \eqref{eq:newdef} in Section~\ref{section:GP}.
Then, motivated by the method in \cite{chihara2}, 
we can choose a suitable 
gauge transformed function 
$\widetilde{W}$, which is formally written by 
$$
\widetilde{W}=\WW+(\Psi_1+\Psi_2)Z_{xx}, 
$$
where $\Psi_1$ and $\Psi_2$ will be defined  
in Remark~\ref{remark:comm} in Section~\ref{section:proof}, 
and they behave as pseudo-differential operators of order $-2$. 
We can handle as $\TW=\WW+\mathcal{O}(|Z|)$ in many situations.  
Moreover, we stress that the commutator of $\Psi_1$ (resp. $\Psi_2$) 
and the principal part $a\p_x^2\{J(U)\p_x^2\}$ effectively works 
to eliminate the part
$R(\p_x^2\WW,J(U)U_x)U_x$ (resp. $R(J(U)\UU,U_x)\p_x\WW$). 
By using them, we can get the desired energy estimate for 
$\widetilde{D}(t)^2
:=\|Z(t)\|_{L^2}^2+\|Z_x(t)\|_{L^2}^2+\|\TW(t)\|_{L^2}^2$ 
so that we obtain $Z=0$, which implies $u=v$. 
Here, $\|\cdot\|_{L^2}$ denotes the standard $L^2$-norm 
for $\RR^d$-valued functions on $\TT$ 
(see Section~\ref{subsection:energy}). 
The assumption $k\geqslant 5$ 
comes from the requirement for the 
energy estimate to make sense,   
which slightly improves the previous one $k\geqslant 6$ 
imposed in \cite{onodera2, onodera3}.
\begin{remark}
The extrinsic approach to uniqueness results 
via the isometric embedding 
into $\RR^d$ has been adopted in broad range of 
geometric PDEs for maps into manifolds:  
Harmonic (or Biharmonic) map heat flow equation, 
Wave (or Biwave) map equation, 
Schr\"odinger map flow equation and 
the third- or fourth-order analogous dispersive curve flow equations, 
and so on. See, e.g., 
\cite{CO, CO2, DW1998, ES, HLSS, Lamm, onodera4, 
ShSt, SW2013} 
and references therein. 
In particular, we mention again that 
the result in \cite{onodera4} shows Theorem~\ref{theorem:uniqueness} 
if we assume that $(N,J,h)$ is a compact Riemann surface 
with constant Gaussian curvature $S$. 
However, the argument of the proof breaks down without the assumption,  
since it is essentially based on the property 
$$
R^N(Y_1,Y_2)Y_3=S\left(
h(Y_2,Y_3)Y_1-h(Y_2,Y_3)Y_1
\right)
\quad 
(Y_1,Y_2,Y_3\in \Gamma(u^{-1}TN)).
$$ 
\end{remark}
\begin{remark}
The choice of $\TW$ is actually crucial in our proof 
as it is the part we exploit the good structure of \eqref{eq:pde}. 
However, before doing this, we require obtaining 
Proposition~\ref{proposition:UU}, \eqref{eq:WW}, and  
the estimate \eqref{eq:momo}, which is more delicate issue, 
since we need to drop all the seemingly bad terms other than 
$R(\p_x^2\WW,J(U)U_x)U_x$ 
and $R(J(U)\UU,U_x)\p_x\WW$ completely. 
To do this, we will develop the basic tools of computations
on $\Gamma(u^{-1}TN)$ to those
on the embedded submanifold in $\RR^d$. 
More concretely, 
we will introduce the notion $R$ in Definition~\ref{definition:newdef} 
and obtain some properties corresponding to those for $R^N$
(see  Proposition~\ref{proposition:R}). 
We point out that we can apply  the assumption 
$\nabla^N R^N=0$ as \eqref{eq:R7}. 
Moreover, we will obtain the connection 
between $R$ and the second fundamental form on $N$ 
in Proposition~\ref{proposition:embcurvature}, 
which will be useful to compare the difference 
$dw_u(a\nabla_x^3u_x)-dw_v(a\nabla_x^3v_x)$ 
in Section~\ref{section:cl}. 
Some other useful geometric properties 
applied in this paper also will be collected in Section~\ref{section:GP}.
\end{remark}
\par
The organization of the paper is as follows: 
In Section~\ref{section:GP}, we collect geometric notion and 
tools of computation used in this paper.
In Section~\ref{section:cl} divided by two subsections,  
we obtain the classical energy estimate for $\WW$ by computing 
the equation satisfied by $\WW$. 
In Section~\ref{section:proof}, based on the results in 
Section~\ref{section:cl}, 
we complete the proof of Theorem~\ref{theorem:uniqueness}.  
\section{Geometric Preliminaries}
\label{section:GP}
In this section, we collect some basic material 
of Riemannian geometry for maps into 
a locally Hermitian symmetric space, 
which will be used in Section~\ref{section:cl} and \ref{section:proof}.
\par 
Let $w$ be an isometric embedding from $(N,J,h)$ into 
some Euclidean space $\RR^d$ with sufficiently large integer $d$. 
For fixed $p\in N$,  
we consider the orthogonal decomposition  
$
\RR^d
=
dw_p(T_pN)
\oplus
\left(
dw_p(T_pN)
\right)^{\perp}
$, 
where 
$dw_p:T_pN\to T_{w{\circ}p}\RR^d\cong \RR^d$ is 
the differential of $w$ at $p\in N$ 
and 
$\left(
dw_p(T_pN)
\right)^{\perp}$
is the orthogonal complement of 
$dw_p(T_pN)$ in $\RR^d$.     
We denote the orthogonal projection mapping 
of $\RR^d$ onto $dw_p(T_pN)$ by $P(w{\circ}p)$ 
and define 
$N(w{\circ}p):=I_d-P(w{\circ}p)$, where 
$I_d$ is the identity mapping on $\RR^d$. 
Moreover, we define $J(w{\circ}p)$  
as an action on $\RR^d$ by first projecting onto 
$dw_p(T_pN)$ 
and then applying the complex structure at $p\in N$. 
More precisely, we define $J(w{\circ}p)$ by 
\begin{align}
J(w{\circ}p)
&:=
dw_p\circ J_{p}\circ dw_{w{\circ}p}^{-1}\circ P(w{\circ}p): 
\RR^d\to dw_p(T_pN). 
\label{eq:J(p)}
\end{align} 
We can extend $P(\cdot)$, $N(\cdot)$, and $J(\cdot)$ 
to a smooth linear operator on $\RR^d$ so that 
$P(q)$, $N(q)$, and $J(q)$ make sense for all $q\in \RR^d$
following the argument in e.g. \cite[pp.17]{NSVZ}. 
Though $J(q)$ is not skew-symmetric and the square is not the minus 
identity in general, 
similar properties still hold if $q\in w(N)$.
Indeed, from the definition of $P(\omega{\circ}p)$ and 
$J(\omega{\circ}p)$, it follows that  
\begin{align}
(P(w{\circ}p)Y_1,Y_2)&=(Y_1,P(w{\circ}p)Y_2),  
\label{eq:J0}
\\
(J(w{\circ}p)Y_1,Y_2)&=-(Y_1,J(w{\circ}p)Y_2),   
\label{eq:J(wp)1}
\\
\left(
J(w{\circ}p)
\right)^2Y_3
&=
-P(w{\circ}p)Y_3, 
\label{eq:J(wp)2}
\end{align}
for any $p\in N$ and $Y_1,Y_2,Y_3\in \RR^d$. 
Here and hereafter, we will denote the inner product 
and the norm in $\RR^d$ by  
$(\cdot,\cdot)$ and $|\cdot|$ respectively.  
\par 
Let $u=u(t,x):[-T,T]\times \TT\to (N,J,h)$ be a 
sufficiently smooth  
map into a $2n$-dimensional K\"ahler manifold $(N,J,h)$, 
and set $U:=w{\circ}u$.
For each $(t,x)\in [-T,T]\times \TT$, 
let $\left\{\nu_{2n+1}, \ldots, \nu_d\right\}$ 
denote a smooth local orthonormal 
frame field for the normal bundle $(dw(TN))^{\perp}$
near $U(t,x)\in w(N)$. 
As is well-known,
for $Y\in \Gamma(u^{-1}TN)$,  
$dw_u(\nabla_xY)$ 
(resp. $dw_u(\nabla_tY)$) is nothing but the 
$dw_u(T_uN)$-component of $\p_x\{dw_u(Y)\}$ 
(resp. $\p_t\{dw_u(Y)\}$), 
that is,   
\begin{align}
dw_u(\nabla_xY)
&=
\p_x
\left\{
dw_u(Y)
\right\}
-
\sum_{k=2n+1}^d
(\p_x\{dw_u(Y)\}, \nu_k(U))\nu_k(U)
\nonumber
\\
&=
\p_x\left\{
dw_u(Y)
\right\}
+
\sum_{k=2n+1}^d
(dw_u(Y), \p_x\left\{\nu_k(U)\right\})\nu_k(U)
\nonumber
\\
&=
\p_x\left\{
dw_u(Y)
\right\}
+
\sum_{k=2n+1}^d
(dw_u(Y), D_k(U)U_x)\nu_k(U)
\nonumber
\\
&=
\p_x\left\{
dw_u(Y)
\right\}
+
A(U)(dw_u(Y),U_x), 
\label{eq:cox2}
\\
dw_u(\nabla_tY)
&=
\p_t\left\{
dw_u(Y)
\right\}
+
A(U)(dw_u(Y),U_t). 
\label{eq:cot2}
\end{align}
where $D_k:=\operatorname{grad} \nu_k$ for $k=2n+1,\dots,d$,  
and $A(q)(\cdot,\cdot):=
\displaystyle\sum_{k=2n+1}^d
(\cdot, D_k(q)\cdot)\nu_k(q)$ 
is the second fundamental form at $q\in w(N)$. 
As is well-known, 
$A(q)(\cdot,\cdot):T_qw(N)\times T_qw(N)\to \RR^d$ 
is a symmetric bilinear map, and each $D_k(q)$ is also symmetric, i.e., 
$(D_k(q)V_1,V_2)=(V_1,D_k(q)V_2)$ holds for any 
$V_1,V_2\in T_qw(N)$. 
Furthermore, we can also see $A(q)(\cdot,\cdot)$ as 
a map on $\RR^d\times \RR^d$ by setting 
$A(q)(V_1,V_2):=A(q)(P(q)V_1,P(q)V_2)$ for any $V_1,V_2\in \RR^d$.  
In addition, we note that $J(U)A(U)(Y_1,Y_2)=0$ holds for any 
$Y_1,Y_2:[-T,T]\times \TT\to \RR^d$. 
This comes from the property
\begin{equation}
J(U)\nu_k(U)=0
\qquad (k=2n+1,\ldots,d), 
\label{eq:J1}
\end{equation}
which is obvious by definition.   
In what follows, 
we will often use these properties and sometimes 
without any comments. 
Moreover, 
we will sometimes use the expression $\displaystyle\sum_{k}$ 
and  $\displaystyle\sum_{k,\ell}$
instead of 
$\displaystyle\sum_{k=2n+1}^{d}$
and 
$\displaystyle\sum_{k=2n+1}^{d}
\displaystyle\sum_{\ell=2n+1}^{d}$
respectively. 
Any confusion will not occur.
\par
We define the operator $\p_x(J(U))$ by 
$\p_x(J(U))Y:=\p_x\{J(U)Y\}-J(U)\p_xY$ 
for any $Y:[-T,T]\times \TT\to \RR^d$. 
The following proposition 
comes from the K\"ahler condition 
$\nabla^N J=0$ on $(N,J,h)$. 
\begin{proposition}
\label{proposition:kaehler}
For any $Y:[-T,T]\times \TT\to \RR^d$, we have 
\begin{align}
\p_x(J(U))Y
&=
\sum_{k}
\left(Y,J(U)D_k(U)U_x\right)\nu_k(U)
-\sum_{k}
(Y,\nu_k(U))J(U)D_k(U)U_x
\nonumber
\\
&=
-A(U)(J(U)Y,U_x)
-\sum_{k}
(Y,\nu_k(U))J(U)D_k(U)U_x.
\label{eq:kaehler2}
\end{align}
\end{proposition}
The proof is given in \cite[Lemma~3.3]{onodera4}, 
and thus we omit the detail.
\par
We next define the operator $\p_x^m(A(U))$ for $m=1,2,\ldots$ 
inductively by 
\begin{align}
\p_x^m(A(U))(Y_1,Y_2)
&:=\p_x\{\p_x^{m-1}(A(U))(Y_1,Y_2)\}
-\p_x^{m-1}(A(U))(\p_xY_1,Y_2)
\nonumber
\\
&\quad
-\p_x^{m-1}(A(U))(Y_1,\p_xY_2)
\nonumber
\end{align} 
for any $Y_1,Y_2:[-T,T]\times \TT\to \RR^d$.
\begin{proposition}
\label{proposition:sf0}
For any $Y_1,Y_2:[-T,T]\times \TT\to \RR^d$, we have 
\begin{align}
&\p_x(A(U))(Y_1,Y_2)
=
\sum_k
(P(U)Y_1,\p_x(D_k(U)P(U))Y_2)\nu_k(U)
\nonumber
\\
&\quad
\phantom{\p_x(A(U))(Y_1,Y_2)}
+
\sum_k(\p_x(P(U))Y_1,D_k(U)P(U)Y_2)\nu_k(U)
\nonumber
\\&\quad
\phantom{\p_x(A(U))(Y_1,Y_2)}
+
\sum_k(P(U)Y_1,D_k(U)P(U)Y_2)D_k(U)U_x,
\label{eq:A1}
\\
&J(U)\p_x(A(U))(Y_1,Y_2)=
\sum_k(P(U)Y_1,D_k(U)P(U)Y_2)J(U)D_k(U)U_x.
\label{eq:A2}
\end{align}
\end{proposition}
\begin{proof}[Proof of Proposition~\ref{proposition:sf0}]
By a simple computation based on 
the definition of $\p_x^m(A(U))$ and the property 
\eqref{eq:J1},  we can easily check 
\eqref{eq:A1} and  \eqref{eq:A2}
hold. 
\end{proof}
Furthermore, 
by the same computation as we obtain 
Proposition~\ref{proposition:sf0}
and by \eqref{eq:kaehler2}, 
we get the following:
\begin{proposition}
\label{proposition:sf}
For any $Y_1,Y_2:[-T,T]\times \TT\to \RR^d$, 
it follows that  
\begin{align}
&\p_x(J(U))\p_x(A(U))(Y_1,Y_2)
\nonumber
\\
&=
-\sum_k(P(U)Y_1,\p_x(D_k(U)P(U))Y_2)J(U)D_k(U)U_x
\nonumber
\\
&\quad
-\sum_k(\p_x(P(U))Y_1,D_k(U)P(U)Y_2)J(U)D_k(U)U_x
\nonumber
\\
&\quad 
-\sum_{k,\ell}(P(U)Y_1,D_k(U)P(U)Y_2)(D_k(U)U_x,\nu_{\ell}(U))
J(U)D_{\ell}(U)U_{\ell}
\nonumber
\\&\quad
-\sum_k A(U)(
(P(U)Y_1,D_k(U)P(U)Y_2)J(U)D_k(U)U_x,U_x),
\label{eq:A3}
\\
&J(U)\p_x^2(A(U))(Y_1,Y_2)
\nonumber
\\
&=
2\sum_k(P(U)Y_1,\p_x(D_k(U)P(U))Y_2)J(U)D_k(U)U_x
\nonumber
\\&\quad
+
2\sum_k(\p_x(P(U))Y_1,D_k(U)P(U)Y_2)J(U)D_k(U)U_x
\nonumber
\\
&\quad
+\sum_k(P(U)Y_1,D_k(U)P(U)Y_2)J(U)\p_x\{D_k(U)U_x\}. 
\label{eq:A4}
\end{align}
\end{proposition}
The following properties on the Riemann curvature tensor $R^N$ 
on $(N,J,h)$
are well-known in Riemannian geometry:  
\begin{proposition}
\label{proposition:R^N}
For any $Y_1, \ldots, Y_4\in \Gamma(u^{-1}TN)$, 
the following properties hold.
\begin{enumerate}
\item[(i)] $R^N(Y_1,Y_2)=-R^N(Y_2,Y_1)$,
\item[(ii)] $h(R^N(Y_1,Y_2)Y_3,Y_4)=h(R^N(Y_3,Y_4)Y_1,Y_2)
=h(R^N(Y_4,Y_3)Y_2,Y_1)$,
\item[(iii)] $R^N(Y_1,Y_2)Y_3+R^N(Y_2,Y_3)Y_1+R^N(Y_3,Y_1)Y_2=0$,
\item[(iv)] $R^N(Y_1,Y_2)J_uY_3=J_u\,R^N(Y_1,Y_2)Y_3$, 
\item[(v)] 
$R^N(J_uY_1,J_uY_2)Y_3=R^N(Y_1,Y_2)Y_3$,
\item[(vi)]
$R^N(J_uY_1,Y_2)Y_3=-R^N(Y_1,J_uY_2)Y_3=R^N(J_uY_2,Y_1)Y_3$,
\item[(vii)]
$\nabla_x
\left\{
R^N(Y_1,Y_2)Y_3
\right\}
=
R^N(\nabla_xY_1,Y_2)Y_3
+
R^N(Y_1,\nabla_xY_2)Y_3
+
R^N(Y_1,Y_2)\nabla_xY_3$.
\end{enumerate} 
\end{proposition}
Note that the property (vii) comes from 
$\nabla^N R^N=0$ imposed on $(N,J,h)$. 
These properties were effectively applied 
to show the local existence of the solution to 
\eqref{eq:pde}-\eqref{eq:data} in \cite{onodera4}. 
On the other hand, in order to show the uniqueness,   
we will require using these properties 
after acting $dw_u$ on them.
For this purpose,
we now introduce the following:
\begin{definition}
\label{definition:newdef}
We define the operator $R$ by 
\begin{align}
R(Y_1,Y_2)Y_3
&:=
dw_u
\left(
R^N(dw_U^{-1}(P(U)Y_1), dw_U^{-1}(P(U)Y_2))
dw_U^{-1}(P(U)Y_3)
\right)
\label{eq:newdef}
\end{align}
for any $Y_1, Y_2, Y_3:[-T,T]\times \TT\to \RR^d$. 
\end{definition}
By the definition, it is immediate to see 
\begin{align}
&R(P(U)\cdot,\cdot)=R(\cdot,P(U)\cdot)
=R(\cdot,\cdot)P(U)=R(\cdot,\cdot),
\label{eq:RP}
\\ 
&R(A(U)(\cdot,\cdot),\cdot)=
R(\cdot, A(U)(\cdot,\cdot))=
R(\cdot,\cdot)A(U)(\cdot,\cdot)
=0.
\label{eq:RA}
\end{align} 
On the local expression of $R$, 
let us recall the following: 
\begin{proposition}
\label{proposition:curvature3}
For any $\Xi_1,\Xi_2,\Xi_3\in \Gamma(u^{-1}TN)$, 
the following relation holds.
\begin{align}
\nonumber
&dw_u
\left(
R^N(\Xi_1,\Xi_2)\Xi_3
\right)
=
\sum_{k=2n+1}^d
\left(
dw_u(\Xi_3), D_k(U)dw_u(\Xi_2)
\right)
P(U)D_k(U)dw_u(\Xi_1)
\\
&\phantom{dw_u
\left(
R^N(Y_1,Y_2)Y_3
\right)}
\quad 
-
\sum_{k=2n+1}^d
\left(
dw_u(\Xi_3), D_k(U)dw_u(\Xi_1)
\right)
P(U)D_k(U)dw_u(\Xi_2).
\nonumber 
\end{align}
\end{proposition}
Proposition~\ref{proposition:curvature3} 
is a kind of expression of 
the Gauss-Codazzi formula in Riemannian geometry, 
and the proof can be seen in \cite[Lemma~3.5]{onodera3}. 
Using Proposition~\ref{proposition:curvature3} with 
$\Xi_i=dw_U^{-1}(P(U)Y_i)\in \Gamma(u^{-1}TN)$, 
$i=1,2,3$, 
we have the following
\begin{proposition}
\label{proposition:embcurvature}
For any $Y_1, Y_2, Y_3:[-T,T]\times \TT\to \RR^d$, 
it follows that
\begin{align}
R(Y_1,Y_2)Y_3
&:=
\sum_{k=2n+1}^d
\left(
P(U)Y_3, D_k(U)P(U)Y_2
\right)
P(U)D_k(U)P(U)Y_1
\nonumber
\\
&
\quad 
-
\sum_{k=2n+1}^d
\left(
P(U)Y_3, D_k(U)P(U)Y_1
\right)
P(U)D_k(U)P(U)Y_2.
\label{eq:embcurvature}
\end{align}
\end{proposition}
If $\Xi_1,\Xi_2, \Xi_3\in \Gamma(u^{-1}TN)$, 
then \eqref{eq:newdef} with $P(U)Y_i=Y_i=dw_u(\Xi_i)$, 
$i=1,2,3$, implies 
\begin{equation}
dw_u
\left(
R^N(\Xi_1,\Xi_2)\Xi_3
\right)=R(dw_u(\Xi_1),dw_u(\Xi_2))dw_u(\Xi_3).
\label{eq:mm1}
\end{equation}
Using them, we have the following 
properties for $R$ corresponding to (i) -(vii) in Proposition~\ref{proposition:R^N} for $R^N$.
\begin{proposition}
\label{proposition:R}
For any $Y_1,Y_2,Y_3,Y_4:[-T,T]\times \TT\to \RR^d$, 
the following properties hold.
\begin{align}
&R(Y_1,Y_2)=-R(Y_2,Y_1)
\label{eq:R1}
\\
&(R(Y_1,Y_2)Y_3,Y_4)=(R(Y_3,Y_4)Y_1,Y_2)
=(R(Y_4,Y_3)Y_2,Y_1),
\label{eq:R2}
\\
&R(Y_1,Y_2)Y_3+R(Y_2,Y_3)Y_1+R(Y_3,Y_1)Y_2=0,
\label{eq:R3}
\\
&R(Y_1,Y_2)J(U)Y_3=J(U)R(Y_1,Y_2)Y_3, 
\label{eq:R4}
\\
&R(J(U)Y_1,J(U)Y_2)Y_3=R(Y_1,Y_2)Y_3,
\label{eq:R5}
\\
&R(J(U)Y_1,Y_2)Y_3=-R(Y_1,J(U)Y_2)Y_3
=R(J(U)Y_2,Y_1)Y_3,
\label{eq:R6} 
\\
&\p_x\left\{
R(Y_1,Y_2)Y_3
\right\}
\nonumber
\\
&=R(\p_x\{P(U)Y_1\},Y_2)Y_3+R(Y_1,\p_x\{P(U)Y_2\})Y_3
\nonumber
\\
&\quad 
+R(Y_1,Y_2)\p_x\{P(U)Y_3\}
-A(U)(R(Y_1,Y_2)Y_3,U_x).
\label{eq:R7}
\end{align}
\end{proposition}
\begin{proof}[Proof of Proposition~\ref{proposition:R}]
It is easy to see \eqref{eq:R1} follows from the definition of $R$. 
It is obvious that \eqref{eq:R2}, $\ldots$, \eqref{eq:R7}
are obtained by acting $dw_u$ to (ii), $\ldots$, (vii) 
in Proposition~\ref{proposition:R^N} respectively. 
We omit the detail, but state only how to obtain \eqref{eq:R7}, 
as it corresponds to the key property 
(vii) coming from $\nabla^N R^N=0$.
To obtain \eqref{eq:R7}, note first that there exists $\Xi_i\in \Gamma(u^{-1}TN)$ such that 
$dw_u(\Xi_i)=P(U)Y_i$ for each $i=1,2,3$.
From (vii) in Proposition~\ref{proposition:R^N}, it follows that 
$$
\nabla_x
\left\{
R^N(\Xi_1,\Xi_2)\Xi_3
\right\}
=
R^N(\nabla_x\Xi_1,\Xi_2)\Xi_3
+
R^N(\Xi_1,\nabla_x\Xi_2)\Xi_3
+
R^N(\Xi_1,\Xi_2)\nabla_x\Xi_3.
$$
Using \eqref{eq:cox2} with 
$Y=R^N(\Xi_1,\Xi_2)\Xi_3$, 
\eqref{eq:mm1}, 
and \eqref{eq:RP}, 
we have
\begin{align}
&dw_u\left(
\nabla_x
\left\{
R^N(\Xi_1,\Xi_2)\Xi_3
\right\}
\right)
\nonumber
\\
&
=
\p_x\left\{
dw_u\left(
R^N(\Xi_1,\Xi_2)\Xi_3
\right)
\right\}
+A(U)(dw_u\left(
R^N(\Xi_1,\Xi_2)\Xi_3
\right),U_x)
\nonumber
\\
&=
\p_x\left\{
R(dw_u(\Xi_1),dw_u(\Xi_2))dw_u(\Xi_3)
\right\}
+A(U)(R(dw_u(\Xi_1),dw_u(\Xi_2))dw_u(\Xi_3),U_x)
\nonumber
\\
&=
\p_x\left\{
R(P(U)Y_1,P(U)Y_2)P(U)Y_3
\right\}
+A(U)(R(P(U)Y_1,P(U)Y_2)P(U)Y_3,U_x)
\nonumber
\\
&=
\p_x\left\{
 R(Y_1,Y_2)Y_3
 \right\}
 +A(U)(R(Y_1,Y_2)Y_3,U_x).
\nonumber
\end{align}
In the same way, 
we have 
\begin{align}
&dw_u(
R^N(\nabla_x\Xi_1,\Xi_2)\Xi_3
+
R^N(\Xi_1,\nabla_x\Xi_2)\Xi_3
+
R^N(\Xi_1,\Xi_2)\nabla_x\Xi_3
)
\nonumber
\\
&=
R(dw_u(\nabla_x\Xi_1),dw_u(\Xi_2))dw_u(\Xi_3)
+
R(dw_u(\Xi_1),dw_u(\nabla_x\Xi_2))dw_u(\Xi_3)
\nonumber
\\&\quad
+
R(dw_u(\Xi_1),dw_u(\Xi_2))dw_u(\nabla_x\Xi_3)
\nonumber
\\
&=
R(\p_x\left\{dw_u(\Xi_1)\right\}
+A(U)(dw_u(\Xi_1),U_x),dw_u(\Xi_2))dw_u(\Xi_3)
\nonumber
\\&\quad 
+
R(dw_u(\Xi_1),\p_x\left\{dw_u(\Xi_2)\right\}
+A(U)(dw_u(\Xi_2),U_x))
dw_u(\Xi_3)
\nonumber
\\&\quad
+
R(dw_u(\Xi_1),dw_u(\Xi_2))
(\p_x\left\{dw_u(\Xi_3)\right\}
+A(U)(dw_u(\Xi_3),U_x))
\nonumber
\\
&=
R(\p_x\{P(U)Y_1\}+A(U)(P(U)Y_1,U_x),P(U)Y_2)P(U)Y_3
\nonumber
\\
&\quad
+
R(P(U)Y_1,\p_x\{P(U)Y_2\}
+A(U)(P(U)Y_2,U_x))
P(U)Y_3
\nonumber
\\&\quad
+
R(P(U)Y_1,P(U)Y_2)
(\p_x\{P(U)Y_3\}
+A(U)(P(U)Y_3,U_x))
\nonumber
\\
&=
R(\p_x\{P(U)Y_1\},Y_2)Y_3
+
R(Y_1,\p_x\{P(U)Y_2\})Y_3
+
R(Y_1,Y_2)
\p_x\{P(U)Y_3\},
\nonumber
\end{align}
where in the final equality we used \eqref{eq:RA}.
Comparing them, we obtain \eqref{eq:R7}. 
\end{proof}
We next set $\mathcal{U}=dw_u(\nabla_xu_x)$ as given  
in Introduction.
\begin{definition}
\label{definition:B_i}
We define operators $B_i(U)$, $i=1,2,3$, by 
\begin{align}
B_1(U)Y
&:=
R(Y,\mathcal{U})J(U)U_x-R(Y,U_x)J(U)\mathcal{U},
\nonumber
\\
B_2(U)Y
&:=
R(Y,\mathcal{U})J(U)U_x+
R(Y,J(U)U_x)\mathcal{U},
\nonumber
\\
B_3(U)Y
&:=
R(Y,U_x)J(U)\mathcal{U}
+R(Y,J(U)\mathcal{U})U_x
\nonumber
\end{align}
for any $Y:[-T,T]\times \TT\to \RR^d$. 
\end{definition}
\begin{proposition}
\label{proposition:B_isym}
Each of 
$B_i(U)$ \ ($i=1,2,3$) is symmetric, that is, 
\begin{equation}
(B_i(U)Y,Z)=(Y,B_i(U)Z)
\label{eq:symmm}
\end{equation}
holds for any $Y,Z:[-T,T]\times \TT\to \RR^d$. 
\end{proposition}
\begin{proof}[Proof of Proposition~\ref{proposition:B_isym}]
For $i=1$ , it follows that 
\begin{align}
(B_1(U)Y,Z)
&=
(R(Y,\mathcal{U})J(U)U_x,Z)
-(R(Y,U_x)J(U)\mathcal{U},Z)
\nonumber
\\
&=
-(R(\mathcal{U},J(U)U_x)Y,Z)
-(R(J(U)U_x,Y)\mathcal{U},Z)
\nonumber
\\
&\quad
+(R(U_x,J(U)\mathcal{U})Y,Z)
+(R(J(U)\mathcal{U},Y)U_x,Z)
\quad
(\because (\eqref{eq:R3}))
\nonumber
\\
&=
-(R(J(U)U_x,Y)\mathcal{U},Z)
+(R(J(U)\mathcal{U},Y)U_x,Z)
\quad 
(\because (\eqref{eq:R1}, \eqref{eq:R6}))
\nonumber
\\
&=
(R(Z,\mathcal{U})J(U)U_x,Y)
-(R(Z,U_x)J(U)\mathcal{U},Y)
\quad
(\because (\eqref{eq:R1}, \eqref{eq:R2})
\nonumber
\\
&=(B_1(U)Z,Y).
\nonumber
\end{align}
For $i=2$ and $i=3$, \eqref{eq:symmm}  follows from \eqref{eq:R2}. 
\end{proof}
\begin{proposition}
\label{proposition:Rsym}
For any $Y:[-T,T]\times \TT\to \RR^d$, it follows that 
\begin{align}
R(Y,\mathcal{U})J(U)U_x
&=\frac{1}{2}R(J(U)\mathcal{U},U_x)Y
+\left(
\frac{1}{2}B_1(U)
+\frac{1}{4}B_2(U)
+\frac{1}{4}B_3(U)
\right)Y, 
\label{eq:ppp1}
\\
R(Y,U_x)J(U)\mathcal{U}
&=\frac{1}{2}R(J(U)\mathcal{U},U_x)Y
-\left(
\frac{1}{2}B_1(U)
-\frac{1}{4}B_2(U)
-\frac{1}{4}B_3(U)
\right)Y. 
\label{eq:ppp2}
\end{align}
\end{proposition}
\begin{proof}[Proof of Proposition~\ref{proposition:Rsym}]
Using $B_i(U)$ ($i=1,2,3$), we can rewrite as follows: 
\begin{align}
&4\,R(Y,\mathcal{U})J(U)U_x
\nonumber
\\
&=
2B_1(U)Y
+
2\left\{
R(Y,\mathcal{U})J(U)U_x+R(Y,U_x)J(U)\mathcal{U}
\right\}
\nonumber
\\
&=2B_1(U)Y+B_2(U)Y+B_3(U)Y
+R(Y,\mathcal{U})J(U)U_x-
R(Y,J(U)U_x)\mathcal{U}
\nonumber
\\
&\quad 
+
R(Y,U_x)J(U)\mathcal{U}
-R(Y,J(U)\mathcal{U})U_x. 
\label{eq:notee1}
\end{align}
Here, it follows that 
\begin{align}
&R(Y,\mathcal{U})J(U)U_x+R(Y,U_x)J(U)\mathcal{U}
\nonumber
\\
&=
-R(\mathcal{U},J(U)U_x)Y
-R(J(U)U_x,Y)\mathcal{U}
\nonumber
\\
&\quad
-R(U_x,J(U)\mathcal{U})Y
-R(J(U)\mathcal{U},Y)U_x
\quad
(\because \eqref{eq:R3})
\nonumber
\\
&=2\,R(J(U)\mathcal{U},U_x)Y
+R(Y,J(U)U_x)\mathcal{U}+R(Y,J(U)\mathcal{U})U_x.
\quad
(\because \eqref{eq:R1}, \eqref{eq:R6})
\label{eq:notee2}
\end{align} 
Combining \eqref{eq:notee1} and \eqref{eq:notee2}, we have 
\eqref{eq:ppp1}. 
In the same way, since 
$$
4\,R(Y,U_x)J(U)\mathcal{U}
=
-2B_1(U)Y
+
2\left\{
R(Y,\mathcal{U})J(U)U_x+R(Y,U_x)J(U)\mathcal{U}
\right\},
$$
we have 
\eqref{eq:ppp2}.
\end{proof} 
Next, we set 
\begin{equation}
T(U):=J(U)\p_x(A(U))(\cdot,U_x)
=\sum_k(P(U)\cdot,D_k(U)U_x)J(U)D_k(U)U_x,
\label{eq:T(U)}
\end{equation}
the exact form of which 
has been obtained by \eqref{eq:A2} in Proposition~\ref{proposition:sf}. 
We here observe the expression of $\p_x(T(U))$ defined by 
$\p_x(T(U))Y:=\p_x\{T(U)Y\}-T(U)\p_xY$. 
By the definition and the expression \eqref{eq:A2}, 
it is easy to find 
\begin{align}
\p_x(T(U))Y
&=
\sum_k
(P(U)Y,\p_x\{D_k(U)U_x \})J(U)D_k(U)U_x
\nonumber
\\
&\quad
+\sum_k
(P(U)Y,D_k(U)U_x ) \p_x\{J(U)D_k(U)U_x\}
\nonumber
\\
&\quad
+
\sum_{k}(\p_x(P(U))Y,D_k(U)U_x)J(U)D_k(U)U_x
\label{eq:TU1}
\end{align}
for any $Y:[-T,T]\times \TT\to \RR^d$. 
In fact, as an application of the tools stated above, 
we can obtain the following 
another expression.
\begin{proposition}
\label{proposition:2TU}
For any $Y:[-T,T]\times \TT\to \RR^d$, it follows that 
\begin{align}
\p_x(T(U))Y
&=
-R(J(U)\UU,U_x)Y
-\frac{1}{2}(B_2(U)+B_3(U))Y
\nonumber
\\
&\quad
-J(U)R(U_x,\p_x(N(U))Y)U_x
-A(U)(J(U)R(U_x,Y)U_x,U_x)
\nonumber
\\
&\quad
+\sum_k\p_x\left(
(U_x,D_k(U)U_x)J(U)D_k(U)P(U)
\right)Y.
\label{eq:ppp3}
\end{align}
\end{proposition}
\begin{proof}[Proof of Proposition~\ref{proposition:2TU}]
Using $J(U)=J(U)P(U)$, the symmetry of $D_k(U)$, 
and \eqref{eq:embcurvature} with 
$Y_1=Y_3=U_x=P(U)U_x$ and $Y_2=Y$, we can deduce 
\begin{align}
T(U)Y
&=
J(U)\sum_k
(P(U)Y,D_k(U)U_x)P(U)D_k(U)U_x
\nonumber
\\
&=
J(U)\sum_k
(U_x, D_k(U)P(U)Y)P(U)D_k(U)U_x
\nonumber
\\
&=
J(U)
\left\{
R(U_x,Y)U_x
+
\sum_k
(U_x, D_k(U)U_x)P(U)D_k(U)P(U)Y
\right\}
\nonumber
\\
&=
J(U)R(U_x,Y)U_x
+
\sum_k
(U_x, D_k(U)U_x)J(U)D_k(U)P(U)Y.
\label{eq:he}
\end{align}
By computing 
$\p_x(T(U))Y=\p_x\{T(U)Y\}-T(U)\p_xY$ 
under the expression, 
we can obtain \eqref{eq:ppp3}. 
We demonstrate only the computation related to the first term 
of the right hand side of \eqref{eq:he} here.
To begin with, it follows that
\begin{align}
&
\p_x\{J(U)R(U_x,Y)U_x\}-J(U)R(U_x,\p_xY)U_x
\nonumber
\\
&=
\p_x(J(U))R(U_x,Y)U_x
+
J(U)\p_x\{
R(U_x,Y)U_x
\}
-J(U)R(U_x,\p_xY)U_x.
\label{eq:TU3}
\end{align}
Using \eqref{eq:kaehler2} with $Y=R(U_x,Y)U_x$, 
and noting 
$R(U_x,Y)U_x\perp \nu_k(U)$, 
we see
\begin{align}
\p_x(J(U))R(U_x,Y)U_x
&=
-A(U)(J(U)R(U_x,Y)U_x,U_x).
\label{eq:TU4}
\end{align}
Using \eqref{eq:R7} with $Y_1=Y_3=U_x=P(U)U_x$ 
and $Y_2=Y$, 
we deduce
\begin{align}
(\sharp)
&:=J(U)\p_x\{
R(U_x,Y)U_x
\}
-J(U)R(U_x,\p_xY)U_x
\nonumber
\\
&=
J(U)R(U_{xx},Y)U_x
+J(U)R(U_x,\p_x\{P(U)Y\})U_x
+J(U)R(U_x,Y)U_{xx}
\nonumber
\\
&\quad 
-J(U)A(U)(R(U_x,Y)U_x,U_x)-J(U)R(U_x,\p_xY)U_x
\nonumber
\\
&=
J(U)R(U_{xx},Y)U_x
+J(U)R(U_x,Y)U_{xx}
+J(U)R(U_x,\p_x(P(U))Y)U_x
\nonumber
\\
&\quad 
-J(U)R(U_x,N(U)\p_xY)U_x
-J(U)A(U)(R(U_x,Y)U_x,U_x),
\nonumber
\end{align}
where in the final equality we use $P(U)+N(U)=I_d$.
By \eqref{eq:cox2}, 
we see $\UU=U_{xx}+A(U)(U_x,U_x)$. 
By \eqref{eq:J1}, 
we see $J(U)A(U)(\cdot,\cdot)=0$.
By \eqref{eq:RP} and $P(U)N(U)=0$, 
$R(\cdot, N(U)\cdot)=0$.  
By $P(U)+N(U)=I_d$, we see $\p_x(P(U))=-\p_x(N(U))$.
Using them, and using \eqref{eq:R1} and \eqref{eq:R4}, 
we deduce 
\begin{align}
(\sharp)
&=
J(U)R(\UU,Y)U_x
+J(U)R(U_x,Y)\UU
-
J(U)R(U_x, \p_x(N(U))Y)U_x
\nonumber
\\
&=
-R(Y,\UU)J(U)U_x
-R(Y,U_x)J(U)\UU
-
J(U)R(U_x, \p_x(N(U))Y)U_x.
\nonumber
\end{align}
Furthermore, applying \eqref{eq:ppp1}-\eqref{eq:ppp2} in 
Proposition~\ref{proposition:Rsym}, 
we have 
\begin{align}
(\sharp)
&=
-R(J(U)\UU,U_x)Y
-\frac{1}{2}(B_2(U)+B_3(U))Y
-
J(U)R(U_x, \p_x(N(U))Y)U_x.
\label{eq:TU5}
\end{align}
Substituting \eqref{eq:TU4} and \eqref{eq:TU5} into 
\eqref{eq:TU3}, we obtain the part of the summation of the 
first four terms of the right hand side of \eqref{eq:ppp3}.
\end{proof}
\begin{remark} 
For any $Y:[-T,T]\times \TT\to \RR^d$, 
since $N(U)Y=\sum_k(Y,\nu_k(U))\nu_k(U)$ and 
$A(U)(Y,U_x)=A(U)(P(U)Y,U_x)$, 
we find  
\begin{align}
&\p_x(N(U))Y
(=-\p_x(P(U))Y)
\nonumber
\\
&=
\sum_{\ell}(Y,D_{\ell}(U)U_x)\nu_{\ell}(U)
+
\sum_{\ell}(Y,\nu_{\ell}(U))D_{\ell}(U)U_x
\nonumber
\\
&=
A(U)(Y,U_x)
+\sum_{\ell}(N(U)Y,D_{\ell}(U)U_x)\nu_{\ell}(U)
+\sum_{\ell}(Y,\nu_{\ell}(U))D_{\ell}(U)U_x. 
\label{eq:PN}
\end{align}
\end{remark}
\begin{definition}
\label{definition:spm}
We define operators $S_{+}(U)$ and $S_{-}(U)$ by 
\begin{equation}
S_{\pm}(U)Y
:=\sum_k(U_x,D_k(U)U_x)
P(U)(J(U)D_k(U) \pm D_k(U)J(U))P(U)Y
\label{eq:spm}
\end{equation}
for any $Y:[-T,T]\times \TT\to \RR^d$.
\end{definition}
By \eqref{eq:J(wp)1} and the symmetry of $D_k(U)$, 
it is immediate to see 
$(S_{\pm}(U)Y_1,Y_2)=\mp(Y_1,S_{\pm}Y_2)$
for any $Y_1,Y_2:[-T,T]\times \TT\to \RR^d$. 
In our proof of Theorem\ref{theorem:uniqueness}, 
the skew-symmetry of $S_{+}$ 
and the symmetry of $S_{-}$  will play the crucial role.
Furthermore, let us here look at the structure of 
$S_{\pm}(U)$ in terms of 
$R$, $T(U)$, and the adjoint $(T(U))^{\ast}$ 
which is obviously given by 
\begin{equation}
(T(U))^{\ast}=\sum_k(\cdot,J(U)D_k(U)U_x)P(U)D_k(U)U_x.
\label{eq:Tstar}
\end{equation}
The observation will make the structure of 
the equation satisfied by $\WW$ clear.
For this purpose, let $Y:[-T,T]\times \TT\to \RR^d$ be any given. 
Applying \eqref{eq:embcurvature} 
with $Y_1=Y$, $Y_2=Y_3=U_x$, and using 
\eqref{eq:R4} and the symmetry of $D_k(U)$, we have
\begin{align}
&\sum_k(U_x,D_k(U)U_x)
P(U)J(U)D_k(U)P(U)Y
\nonumber
\\
&=J(U)\sum_k(U_x,D_k(U)U_x)
P(U)D_k(U)P(U)Y
\nonumber
\\
&=
J(U)R(Y,U_x)U_x+J(U)\sum_k(U_x,D_k(U)P(U)Y)
P(U)D_k(U)U_x
\nonumber
\\
&=
R(Y,U_x)J(U)U_x
+\sum_k(P(U)Y,D_k(U)U_x)
J(U)D_k(U)U_x
\nonumber
\\
&=
R(Y,U_x)J(U)U_x+T(U)Y. 
\label{eq:131}
\end{align}
In the same way as above, applying 
\eqref{eq:embcurvature} with 
$Y_1=J(U)Y$, $Y_2=Y_3=U_x$, and using 
\eqref{eq:R6}, the symmetry of $D_k(U)$ 
and \eqref{eq:J(wp)1}, we deduce  
\begin{align}
&\sum_k(U_x,D_k(U)U_x)
P(U)D_k(U)J(U)P(U)Y
\nonumber
\\
&=\sum_k(U_x,D_k(U)U_x)
P(U)D_k(U)P(U)J(U)Y
\nonumber
\\
&=
R(J(U)Y,U_x)U_x
+\sum_k(U_x,D_k(U)J(U)Y)
P(U)D_k(U)U_x
\nonumber
\\
&=-R(Y,J(U)U_x)U_x
-\sum_k(Y, J(U)D_k(U)U_x)
P(U)D_k(U)U_x
\nonumber
\\
&=
-R(Y,J(U)U_x)U_x-(T(U))^{\ast}Y. 
\label{eq:231}
\end{align}
Combining \eqref{eq:131} and \eqref{eq:231}, 
and using  
\eqref{eq:R1} and \eqref{eq:R3}, we find 
\begin{align}
S_{+}(U)
&=
T(U)-(T(U))^{\ast}
+R(\cdot,U_x)J(U)U_x-R(\cdot,J(U)U_x)U_x
\nonumber
\\&=
T(U)-(T(U))^{\ast}
+R(J(U)U_x,U_x), 
\label{eq:331}
\\
S_{-}(U)&=
T(U)+(T(U))^{\ast}
+R(\cdot,U_x)J(U)U_x
+R(\cdot,J(U)U_x)U_x.
\label{eq:431}
\end{align}
It is obvious to see both $T(U)-(T(U))^{\ast}$
and $R(J(U)U_x,U_x)$ are skew-symmetric and 
both $T(U)+(T(U))^{\ast}$
and 
$R(\cdot,U_x)J(U)U_x+R(\cdot,J(U)U_x)U_x$ are symmetric. 
We can observe the skew-symmetry of $S_{+}(U)$ and 
symmetry of $S_{-}(U)$ also from them.
\par
%
%
\section{Equations satisfied by $\UU$ and $\WW$, and 
the energy estimate in $L^2$}
\label{section:cl}
Let $u,v$ be solutions to \eqref{eq:pde}-\eqref{eq:data} 
constructed in Theorem~\ref{theorem:existence}.
Then 
$u$ and $v$ satisfy  
$u_x, v_x 
\in L^{\infty}(-T,T;H^5(\TT;TN))
\cap
C([-T,T];H^{4}(\TT;TN))$. 
We fix $w$ as an isometric embedding of $(N,J,h)$ into 
some Euclidean space $\RR^d$. 
We set 
$U:=w{\circ}u$, 
$V:=w{\circ}v$, 
$Z:=U-V$, 
$\UU:=dw_u(\nabla_xu_x)$, 
$\VV:=dw_v(\nabla_xv_x)$, 
and $\WW:=\UU-\VV$. 
The goal of this section is to derive the equation 
satisfied by $\WW$ and 
the classical estimate for $\WW$ in $L^2(\TT;\RR^d)$ 
permitting loss of derivatives.   
\par 
Before starting them, some comments are in order: 
The equivalence between $u_x, v_x 
\in L^{\infty}(-T,T;H^5(\TT;TN))
\cap
C([-T,T];H^{4}(\TT;TN))$
and 
$U_x, V_x 
\in L^{\infty}(-T,T;H^5(\TT;\RR^d))
\cap
C([-T,T];H^{4}(\TT;\RR^d))$ 
follows from the Gagliardo-Nirenberg inequality.  
Therefore we see 
$$
\UU,\VV,\WW\in L^{\infty}(-T,T;H^4(\TT;\RR^d)) 
\cap C([-T,T];H^{3}(\TT;\RR^d)).$$
Particularly from the Sobolev embedding $H^1(\TT;\RR^d)\subset 
C(\TT;\RR^d)$, we see   
$\p_x^{\ell}U_x, \p_x^{\ell}V_x\in 
C([-T,T]\times \TT;\RR^d)$ 
for $\ell=0,1,\ldots,3$
and $\p_x^{\ell}\UU, \p_x^{\ell}\VV\in 
C([-T,T]\times \TT;\RR^d)$ 
for $\ell=0,1,\ldots,2$. 
They will be used below without any comments.
In addition, 
we will sometimes write $dw$ instead of $dw_u$ or $dw_v$ 
below for simplicity.
\subsection{
Equation
satisfied by $\UU$} 
\label{subsection:UU}
In this subsection, we compute the equation 
satisfied by $\UU$. 
The goal is to derive the following: 
\begin{proposition}
\label{proposition:UU}
Under the setting as above, we have 
\begin{align}
\p_t\UU=L(U)\UU
+\sum_k\mathcal{O}(|\p_x^2\UU|+|\p_x\UU|)\nu_k(U)
+\mathcal{O}(|U|+|U_x|+|\UU|),
\label{eq:UU}
\end{align} 
where $L(U)$ is the partial differential operator which is given by
\begin{align}
L(U)Y
&=
a\p_x^2\{J(U)\p_x^2Y \}
+2aA(U)(J(U)\p_x^3Y,U_x)
+\lambda\p_x\{J(U)\p_xY\}
\nonumber
\\
&\quad
+
(-a+b)\,R(\p_x^2Y,J(U)U_x)U_x
+(-a+b+c)\,\p_x\left\{
R(J(U)U_x,U_x)\p_xY
\right\}
\nonumber
\\
&\quad 
+\left(-\frac{a}{2}-\frac{b}{2}+3c\right)\,R(J(U)\UU,U_x)\p_xY
+\sum_{i=1}^3c_iB_i(U)\p_xY
\nonumber
\\
&\quad
+a\p_x\{S_{+}(U)\p_xY \}
+a\p_x(S_{-}(U))\p_xY
\nonumber
\\
&\quad
+2a\sum_k(\UU,D_k(U)U_x)J(U)D_k(U)\p_xY
\label{eq:UU2}
\end{align}
for any $Y:[-T,T]\times \TT\to \RR^d$.
Here, $c_1, c_2, c_3\in \RR$ are constants depending on 
$a,b,c$, and $B_i(U)$ for $i=1,2,3$ have been defined 
in Definition~\ref{definition:B_i}, and $S_{+}(U)$ and $S_{-}(U)$ 
have been defined in Definition~\ref{definition:spm} or 
\eqref{eq:331}-\eqref{eq:431}.
\end{proposition}
To prove this, lengthy and delicate computations 
will be required. 
For readers who want to know
how Proposition~\ref{proposition:UU} is applied in advance,
it is also recommendable to skip to Subsection~\ref{subsection:energy}. 
\begin{proof}[Proof of Proposition~\ref{proposition:UU}]
To begin with, we compute the equation satisfied by $U$. 
Using \eqref{eq:pde} and \eqref{eq:mm1}, 
we deduce  
\begin{align}
U_t
&=
dw(u_t)
\nonumber
\\
&=
a\,dw(\nabla_xJ_u\nabla_x^2u_x)
+
\lambda\,dw(J_u\nabla_xu_x)
\nonumber
\\
&\quad
+
b\,dw(R^N(\nabla_xu_x,u_x)J_uu_x)
+
c\,dw(R^N(J_uu_x,u_x)\nabla_xu_x)
\nonumber
\\
&=
a\,dw(\nabla_xJ_u\nabla_x^2u_x)
+
\lambda\,J(U)dw(\nabla_xu_x)
\nonumber
\\
&\quad 
+
b\, R(dw(\nabla_xu_x),dw(u_x))dw(J_uu_x)
\nonumber
\\
&\quad 
+
c\, R(dw(J_uu_x),dw(u_x))dw(\nabla_xu_x)
\nonumber
\\
&=
a\,dw(\nabla_xJ_u\nabla_x^2u_x)
+
\lambda\,
J(U)\UU 
+
b\,R(\UU,U_x)J(U)U_x
+
c\,R(J(U)U_x,U_x)\UU.
\label{eq:u1}
\end{align}
In the computation, we also repeatedly use the fact 
$dw(J_u\Xi)=J(U)dw(\Xi)$ holds for any 
$\Xi\in \Gamma(u^{-1}TN)$. 
Note also that, for the first term of the right hand side,  
we use the fact  $J_u\nabla_x^3u_x=\nabla_xJ_u\nabla_x^2u_x$ 
which comes from the K\"ahler condition $\nabla^N J=0$. 
Furthermore, 
using \eqref{eq:cox2} and $J(U)A(U)(\cdot,\cdot)=0$ 
which comes from \eqref{eq:J1}, 
we see
\begin{align}
dw(\nabla_x^2u_x)
&=
\p_x\left\{
dw(\nabla_xu_x)
\right\}
+
A(U)(dw(\nabla_xu_x),U_x)
\nonumber
\\
&
=
\p_x\UU
+
A(U)(\UU,U_x),
\label{eq:u2}
\\
J(U)dw(\nabla_x^2u_x)
&=
J(U)\p_x\UU
+
J(U)A(U)(\UU,U_x)
=
J(U)\p_x\UU, 
\label{eq:u3}
\\
dw(\nabla_xJ_u\nabla_x^2u_x)
&=
\p_x
\left\{
dw(J_u\nabla_x^2u_x)
\right\}
+
A(U)(dw(J_u\nabla_x^2u_x),U_x)
\nonumber
\\
&=
\p_x\left\{
J(U)dw(\nabla_x^2u_x)
\right\}
+
A(U)(J(U)dw(\nabla_x^2u_x), U_x)
\nonumber
\\
&=
\p_x\left\{
J(U)\p_x\UU
\right\}
+
A(U)(J(U)\p_x\UU,U_x). 
\label{eq:u4}
\end{align}
Substituting \eqref{eq:u4} into \eqref{eq:u1}, 
we have 
\begin{align}
U_t
&=
a\,\p_x\left\{
J(U)\p_x\UU
\right\}
+
a\,A(U)(J(U)\p_x\UU,U_x)
+
\mathcal{O}(|U|+|U_x|+|\UU|).
\label{eq:U_t}
\end{align}
On the other hand,
we notice $dw(\nabla_xJ_u\nabla_x^2u_x)=dw(J_u\nabla_x^3u_x)
=J(U)dw(\nabla_x^3u_x)$ 
also follows from the K\"ahler condition $\nabla^N J=0$. 
Using \eqref{eq:cox2} and \eqref{eq:u2}, we obtain 
\begin{align}
dw(\nabla_x^3u_x)
&=
\p_x\left\{
dw(\p_x^2u_x)
\right\}
+
A(U)(dw(\p_x^2u_x),U_x)
\nonumber
\\
&=
\p_x^2\UU
+\p_x\left\{
A(U)(\UU,U_x)
\right\}
+A(U)(\p_x\UU+A(U)(\UU,U_x),U_x)
\nonumber
\\
&=
\p_x^2\UU
+\p_x(A(U))(\UU,U_x)
+2A(U)(\p_x\UU,U_x)
\nonumber
\\
&\quad 
+A(U)(\UU,U_{xx})+A(U)(A(U)(\UU,U_x),U_x).
\label{eq:3mmm}
\end{align}
Hence, using $J(U)A(U)(\cdot,\cdot)=0$ again, we get 
\begin{align}
dw(\nabla_xJ_u\nabla_x^2u_x)
&=
J(U)\p_x^2\UU
+J(U)\p_x(A(U))(\UU,U_x). 
\label{eq:3mm}
\end{align}
Therefore, we notice \eqref{eq:U_t} can be rewritten as 
\begin{equation}
U_t
=a\,J(U)\p_x^2\UU
+\mathcal{O}(|U|+|U_x|+|\UU|).
\label{eq:U_t2}
\end{equation}
\par 
Next, we compute the equation  
satisfied by $\UU$. 
From 
\eqref{eq:pde}, 
\eqref{eq:mm1}, 
and
\eqref{eq:cot2}, 
it follows that 
\begin{align}
\p_t\UU
&=
\p_t(dw(\nabla_xu_x))
\nonumber
\\
&=
dw(\nabla_t\nabla_xu_x)
-
A(U)(dw(\nabla_xu_x),U_t)
\nonumber
\\
&=
dw(\nabla_x^2u_t+R^N(u_t,u_x)u_x)
-
A(U)(\UU,U_t)
\nonumber
\\
&=
dw(\nabla_x^2u_t)
+
R(U_t,U_x)U_x
-
A(U)(\UU,U_t)
\nonumber
\\
&=
a\, dw(\nabla_x^2J_u\nabla_x^3u_x)
+
\lambda\,
dw(\nabla_x^2J_u\nabla_xu_x)
+F_R+
R(U_t,U_x)U_x
-
A(U)(\UU,U_t), 
\label{eq:mUU}
\\
F_R
&:=dw
 \left(
 b\nabla_x^2\left\{
 R^N(\nabla_xu_x,u_x)J_uu_x
 \right\}
 +
 c\,
 \nabla_x^2\left\{
 R^N(J_uu_x,u_x)\nabla_xu_x
 \right\}
 \right).
 \nonumber
\end{align}
We compute each term of the right hand 
side of \eqref{eq:mUU}.
Starting from 
substituting \eqref{eq:U_t2} and using \eqref{eq:R6}, 
we have 
\begin{align}
R(U_t,U_x)U_x
&=
a\,R(J(U)\p_x^2\UU,U_x)U_x
+
\mathcal{O}(|U|+|U_x|+|\UU|)
\nonumber
\\
&=
-a\,R(\p_x^2\UU,J(U)U_x)U_x
+
\mathcal{O}(|U|+|U_x|+|\UU|),
\label{eq:II}
\\
A(U)(\UU,U_t)
&=
a\,A(U)(\UU, J(U)\p_x^2\UU)
+\mathcal{O}(|U|+|U_x|+|\UU|).
\label{eq:III}
\end{align}
Using \eqref{eq:u4} and the K\"ahler condition 
$\nabla^NJ=0$, we immediately see  
\begin{align}
\lambda\,
dw(\nabla_x^2J_u\nabla_xu_x)
&=
\lambda\,
\p_x\left\{
J(U)\p_x\UU
\right\}
+
\lambda\,
A(U)(J(U)\p_x\UU,U_x).
\label{eq:I_2} 
\end{align}
Concerning the term $F_R$, 
we compute using Proposition~\ref{proposition:R^N}. 
Indeed, 
using  (i) and (vii)  in Proposition~\ref{proposition:R^N}, 
we deduce 
\begin{align}
F_{R1}
&
:=\nabla_x^2\left\{
R^N(\nabla_xu_x,u_x)J_uu_x
\right\}
\nonumber
\\
&=
R^N(\nabla_x^3u_x,u_x)J_uu_x
+R^N(\nabla_xu_x,\nabla_x^2u_x)J_uu_x
\nonumber
\\
&\quad 
+R^N(\nabla_xu_x,u_x)J_u\nabla_x^2u_x
+2R^N(\nabla_x^2u_x,\nabla_xu_x)J_uu_x
\nonumber
\\
&\quad
+2R^N(\nabla_x^2u_x,u_x)J_u\nabla_xu_x
+2R^N(\nabla_xu_x,\nabla_x u_x)J_u\nabla_xu_x
\nonumber
\\
&=
R^N(\nabla_x^3u_x,u_x)J_uu_x
+R^N(\nabla_x^2u_x,\nabla_xu_x)J_uu_x
\nonumber
\\
&\quad 
+2R^N(\nabla_x^2u_x,u_x)J_u\nabla_xu_x
+R^N(\nabla_xu_x,u_x)J_u\nabla_x^2u_x.
\nonumber
\end{align}
Then, using (iii) in Proposition~\ref{proposition:R^N}, 
we have 
\begin{align}
F_{R1}
&=-R^N(u_x,J_uu_x)\nabla_x^3u_x
-R^N(J_uu_x,\nabla_x^3u_x)u_x
-R^N(\nabla_xu_x,J_uu_x)\nabla_x^2u_x
\nonumber
\\
&\quad
-R^N(J_uu_x,\nabla_x^2u_x)\nabla_xu_x
-2R^N(u_x,J_u\nabla_xu_x)\nabla_x^2u_x
\nonumber
\\
&\quad 
-2R^N(J_u\nabla_xu_x,\nabla_x^2u_x)u_x
-R^N(u_x,J_u\nabla_x^2u_x)\nabla_xu_x
-R^N(J_u\nabla_x^2u_x,\nabla_xu_x)u_x.
\nonumber
\end{align}
Furthermore, using (i) and (vi) in Proposition~\ref{proposition:R^N},  
we get 
\begin{align}
F_{R1}
&=
R^N(\nabla_x^3u_x,J_uu_x)u_x
+R^N(J_uu_x,u_x)\nabla_x^3u_x
\nonumber
\\
&\quad 
-3R^N(\nabla_xu_x,J_uu_x)\nabla_x^2u_x
+3R^N(\nabla_x^2u_x,J_u\nabla_xu_x)u_x.
\nonumber
\end{align}
In the same way, 
using  (vii)  in Proposition~\ref{proposition:R^N},
we deduce 
\begin{align}
F_{R2}
&
:=\nabla_x^2\left\{
R^N(J_uu_x,u_x)\nabla_xu_x
\right\}
\nonumber
\\
&=
R^N(J_u\nabla_x^2u_x,u_x)\nabla_xu_x
+R^N(J_uu_x, \nabla_x^2u_x)\nabla_xu_x
\nonumber
\\
&\quad
+2R^N(J_u\nabla_xu_x,\nabla_xu_x)\nabla_xu_x
+2R^N(J_u\nabla_xu_x,u_x)\nabla_x^2u_x
\nonumber
\\
&\quad
+2R^N(J_uu_x,\nabla_xu_x)\nabla_x^2u_x
+R^N(J_uu_x,u_x)\nabla_x^3u_x.
\nonumber
\end{align}
Then, using (i) and (vi)  in Proposition~\ref{proposition:R^N}, 
we have
\begin{align}
F_{R2}
&=
R^N(J_uu_x,u_x)\nabla_x^3u_x
-4R^N(\nabla_xu_x,J_uu_x)\nabla_x^2u_x
\nonumber
\\
&\quad 
-2R^N(\nabla_x^2u_x,J_uu_x)\nabla_xu_x
+2R^N(J_u\nabla_xu_x,\nabla_xu_x)\nabla_xu_x.
\nonumber
\end{align}
Substituting them into $F_R=dw(bF_{R1}+cF_{R2})$, we deduce 
\begin{align}
F_R&=
dw\Bigl(
b\,R^N(\nabla_x^3u_x,J_uu_x)u_x
+(b+c)\,R^N(J_uu_x,u_x)\nabla_x^3u_x
\nonumber
\\
&\quad 
-(3b+4c)\,R^N(\nabla_xu_x,J_uu_x)\nabla_x^2u_x
+3b\,R^N(\nabla_x^2u_x,J_u\nabla_xu_x)u_x
\nonumber
\\
&\quad 
-2c\,R^N(\nabla_x^2u_x,J_uu_x)\nabla_xu_x
+2c\,R^N(J_u\nabla_xu_x,\nabla_xu_x)\nabla_xu_x
\Bigr)
\nonumber
\\
&=
b\,R(dw( \nabla_x^3u_x),J(U)U_x)U_x
+(b+c)\,R(J(U)U_x,U_x)dw( \nabla_x^3u_x)
\nonumber
\\
&\quad 
-(3b+4c)\,R(\UU,J(U)U_x)dw(\nabla_x^2u_x)
+3b\,R(dw(\nabla_x^2u_x),J(U)\UU)U_x
\nonumber
\\
&\quad 
-2c\,R(dw(\nabla_x^2u_x),J(U)U_x)\UU
+2c\,R(J(U)\UU,\UU)\UU.
\nonumber
\end{align}
Substituting \eqref{eq:u2} and \eqref{eq:3mmm} into above, 
and noting \eqref{eq:RA} holds, 
we obtain 
\begin{align}
F_R 
&=
b\,R(\p_x^2\UU,J(U)U_x)U_x
+(b+c)\,R(J(U)U_x,U_x)\p_x^2\UU
\nonumber
\\
&\quad
+2b\,R(A(U)(\p_x\UU,U_x),J(U)U_x)U_x
+2(b+c)\,R(J(U)U_x,U_x)A(U)(\p_x\UU,U_x)
\nonumber
\\
&\quad 
-(3b+4c)\,R(\UU,J(U)U_x)\p_x\UU
+3b\,R(\p_x\UU,J(U)\UU)U_x
\nonumber
\\
&\quad 
-2c\,R(\p_x\UU,J(U)U_x)\UU
+\mathcal{O}(|U|+|U_x|+| \UU |)
\nonumber
\\
&=
b\,R(\p_x^2\UU,J(U)U_x)U_x
+(b+c)\,R(J(U)U_x,U_x)\p_x^2\UU
\nonumber
\\
&\quad 
-(3b+4c)\,R(\UU,J(U)U_x)\p_x\UU
+3b\,R(\p_x\UU,J(U)\UU)U_x
\nonumber
\\
&\quad 
-2c\,R(\p_x\UU,J(U)U_x)\UU
+\mathcal{O}(|U|+|U_x|+| \UU |).
\label{eq:mI_3}
\end{align}
Furthermore,  using \eqref{eq:R7} 
with $Y_1=J(U)U_x$, $Y_2=U_x$, $Y_3=\p_x\UU$, 
we have 
\begin{align}
&\p_x
\left\{
R(J(U)U_x,U_x)\p_x\UU
\right\}
\nonumber
\\
&=
R(\p_x\left\{J(U)U_x\right\},U_x)\p_x\UU 
+R(J(U)U_x,U_{xx})\p_x\UU
\nonumber
\\
&\quad
+R(J(U)U_x,U_x)\p_x\{
P(U)\p_x\UU\}
-A(U)(R(J(U)U_x,U_x)\p_x\UU, U_x).
\nonumber
\end{align}
Hence, using $\p_x\{
P(U)\p_x\UU\}=P(U)\p_x^2\UU+\p_x(P(U))\p_x\UU$ 
and \eqref{eq:RP}, we obtain 
\begin{align}
&R(J(U)U_x,U_x)\p_x^2\UU
\nonumber
\\
&=
\p_x
\left\{
R(J(U)U_x,U_x)\p_x\UU
\right\}
-R(\p_x\left\{J(U)U_x\right\},U_x)\p_x\UU
\nonumber
\\
&\quad 
-R(J(U)U_x,U_{xx})\p_x\UU
-R(J(U)U_x,U_x)\p_x(P(U))\p_x\UU
\nonumber
\\
&\quad
+A(U)(R(J(U)U_x,U_x)\p_x\UU, U_x).
\nonumber
\end{align}
Here, from \eqref{eq:kaehler2} and \eqref{eq:J1}, 
we see  
\begin{align}
\p_x\left\{
J(U)U_x
\right\}
&=\p_x(J(U))U_x+J(U)U_{xx}
=
-A(U)(J(U)U_x,U_x)+J(U)\UU. 
\label{eq:6mm}
\end{align}
In addition, recalling \eqref{eq:PN} and 
using $(\p_x\UU,\nu_k(U))=-(\UU,D_k(U)U_x)$ which comes from 
$\UU\perp \nu_k(U)$, 
we have 
\begin{align}
\p_x(P(U))\p_x\UU
&=
-A(U)(\p_x\UU,U_x)
-\sum_k(N(U)\p_x\UU,D_k(U)U_x)\nu_k(U)
\nonumber
\\
&\quad
-\sum_k(\p_x\UU,\nu_k(U))D_k(U)U_x
\nonumber
\\
&=
-A(U)(\p_x\UU,U_x)+\mathcal{O}(|U|+|U_x|+|\UU|). 
\label{eq:7mo}
\end{align}
Substituting \eqref{eq:6mm} and \eqref{eq:7mo}, and then  
using $U_{xx}+A(U)(U_x,U_x)=\UU$, \eqref{eq:R6}, 
and \eqref{eq:RA}, 
we deduce 
\begin{align}
&R(J(U)U_x,U_x)\p_x^2\UU
\nonumber
\\
&=
\p_x
\left\{
R(J(U)U_x,U_x)\p_x\UU
\right\}
-R(J(U)\UU,U_x)\p_x\UU
\nonumber
\\
&\quad 
-R(J(U)U_x,U_{xx})\p_x\UU
+R(J(U)U_x,U_x)A(U)(\p_x\UU,U_x)
\nonumber
\\
&\quad
+A(U)(R(J(U)U_x,U_x)\p_x\UU, U_x)
+
\mathcal{O}(|U|+|U_x|+|\UU|)
\nonumber
\\
&=
\p_x
\left\{
R(J(U)U_x,U_x)\p_x\UU
\right\}
-2R(J(U)\UU,U_x)\p_x\UU
\nonumber
\\
&\quad
+A(U)(R(J(U)U_x,U_x)\p_x\UU, U_x)
+
\mathcal{O}(|U|+|U_x|+|\UU|). 
\label{eq:8mmm}
\end{align}
Substituting \eqref{eq:8mmm} into \eqref{eq:mI_3} 
and then  
using \eqref{eq:R6},
we obtain 
\begin{align}
F_R 
&=
b\,R(\p_x^2\UU,J(U)U_x)U_x
+(b+c)\,\p_x\left\{
R(J(U)U_x,U_x)\p_x\UU
\right\}
\nonumber
\\
&\quad 
+(b+2c)\,R(J(U)\UU,U_x)\p_x\UU
+3b\,R(\p_x\UU,J(U)\UU)U_x
-2c\,R(\p_x\UU,J(U)U_x)\UU
\nonumber
\\
&\quad
+(b+c)A(U)(R(J(U)U_x,U_x)\p_x\UU, U_x)
+\mathcal{O}(|U|+|U_x|+| \UU |).
\label{eq:I_333}
\end{align} 
We observe the structure of $F_R$ in more
detail by  
using \eqref{eq:ppp1} and 
\eqref{eq:ppp2} 
in Proposition~\ref{proposition:Rsym}. 
Indeed, using \eqref{eq:R3} with $Y_1=\p_x\UU$, $Y_2=J(U)\UU$,
$Y_3=U_x$ and \eqref{eq:R1},  
and then using \eqref{eq:ppp2}
with $Y=\p_x\UU$,  
we deduce 
\begin{align}
&R(\p_x\UU,J(U)\UU)U_x
\nonumber
\\
&=-R(J(U)\UU,U_x)\p_x\UU-R(U_x,\p_x\UU)J(U)\UU
\nonumber
\\
&=-R(J(U)\UU,U_x)\p_x\UU+R(\p_x\UU,U_x)J(U)\UU
\nonumber
\\
&=
-\frac{1}{2}R(J(U)\UU,U_x)\p_x\UU
-\left(\frac{B_1(U)}{2}-\frac{B_2(U)}{4}-\frac{B_3(U)}{4}\right)
\p_x\UU.
\label{eq:7mm}
\end{align}
In the same way, using \eqref{eq:ppp1}, we can obtain 
\begin{align}
&R(\p_x\UU,J(U)U_x)\UU
\nonumber
\\
&=
-\frac{1}{2}R(J(U)\UU,U_x)\p_x\UU
+\left(\frac{B_1(U)}{2}+\frac{B_2(U)}{4}+\frac{B_3(U)}{4}\right)
\p_x\UU.
\label{eq:8mm}
\end{align}
Substituting \eqref{eq:7mm} and \eqref{eq:8mm},  
we have
\begin{align}
F_R
&=
b\,R(\p_x^2\UU,J(U)U_x)U_x
+(b+c)\,\p_x\left\{
R(J(U)U_x,U_x)\p_x\UU
\right\}
\nonumber
\\
&\quad 
+\left(-\frac{b}{2}+3c\right)\,R(J(U)\UU,U_x)\p_x\UU
+\sum_{i=1}^3b_iB_i(U)\p_x\UU
\nonumber
\\
&\quad
+(b+c)A(U)(R(J(U)U_x,U_x)\p_x\UU, U_x)
+\mathcal{O}(|U|+|U_x|+| \UU |),
\label{eq:I_3}
\end{align}
where $b_1=-\frac{3b}{2}-c$, 
$b_2=b_3=\frac{3b}{4}-\frac{c}{2}$, though the constants are now 
not important.
Substituting 
\eqref{eq:II}, 
\eqref{eq:III}, 
\eqref{eq:I_2}
and \eqref{eq:I_3}
into \eqref{eq:mUU}, 
we arrive at the following. 
\begin{align}
\p_t\UU
&=
a\,dw(\nabla_x^2J_u\nabla_x^3u_x)
+
\lambda\,
\p_x\left\{
J(U)\p_x\UU
\right\}
\nonumber
\\
&\quad
+
(-a+b)\,R(\p_x^2\UU,J(U)U_x)U_x
+(b+c)\,\p_x\left\{
R(J(U)U_x,U_x)\p_x\UU
\right\}
\nonumber
\\
&\quad 
+\left(-\frac{b}{2}+3c\right)\,R(J(U)\UU,U_x)\p_x\UU
+\sum_{i=1}^3b_iB_i(U)\p_x\UU
\nonumber
\\
&\quad
+r(U,U_x,\UU,\p_x\UU,\p_x^2\UU)
+\mathcal{O}(|U|+|U_x|+| \UU |),
\label{eq:1UU}
\end{align}
where 
\begin{align}
r(U,U_x,\UU,\p_x\UU,\p_x^2\UU)
&=\lambda\,
A(U)(J(U)\p_x\UU,U_x)
-a\,A(U)(\UU,J(U)\p_x^2\UU)
\nonumber
\\
&\quad
+(b+c)A(U)(R(J(U)U_x,U_x)\p_x\UU, U_x)
\nonumber
\\
&=
\sum_k\mathcal{O}(|\p_x^2\UU|+|\p_x\UU|)\nu_k(U).
\label{eq:r}
\end{align}
\par 
Let us now move on to the computation of the first term of 
the right hand side of 
\eqref{eq:1UU}. 
We set $I(U):=dw(\nabla_x^2J_u\nabla_x^3u_x)$. 
Starting from using \eqref{eq:cox2} twice, 
\begin{align}
I(U)&=
\p_x
\left\{
dw(\nabla_xJ_u\nabla_x^3u_x)
\right\}
+A(U)(dw(\nabla_xJ_u\nabla_x^3u_x), U_x)
\nonumber
\\
&=
\p_x
\left\{
\p_x\{dw(J_u\nabla_x^3u_x)\}
+
A(U)(dw(J_u\nabla_x^3u_x),U_x)
\right\}
\nonumber
\\
&\quad
+A(U)(\p_x\{dw(J_u\nabla_x^3u_x)\}
+
A(U)(dw(J_u\nabla_x^3u_x),U_x), U_x).
\nonumber
\end{align}
We have already obtained  
$dw(J_u\nabla_x^3u_x)
=J(U)\p_x^2\UU+J(U)\p_x(A(U))(\UU,U_x)$ 
by \eqref{eq:3mm}. Substituting this into the above,
we have  
\begin{align}
I(U)&=
\p_x^2\{ J(U)\p_x^2\UU\}
+\p_x^2
\{
J(U)\p_x(A(U))(\UU,U_x)
\}
\nonumber
\\
&
\quad
+\p_x\{
A(U)(J(U)\p_x^2\UU+J(U)\p_x(A(U))(\UU,U_x),U_x)
\}
\nonumber
\\
&\quad
+A(U)(\p_x\{
J(U)\p_x^2\UU+J(U)\p_x(A(U))(\UU,U_x)
\},U_x)
\nonumber
\\
&\quad
+A(U)(A(U)(J(U)\p_x^2\UU+J(U)\p_x(A(U))(\UU,U_x),U_x),U_x).
\nonumber
\end{align}
By a simple computation,  we see
\begin{align}
I(U)&=
\p_x^2\{ J(U)\p_x^2\UU\}
+2A(U)(J(U)\p_x^3\UU,U_x)
\nonumber
\\
&\quad
+
J(U)\p_x(A(U))(\p_x^2\UU,U_x)
+
\p_x(A(U))(J(U)\p_x^2\UU,U_x)
\nonumber
\\
&\quad
+2J(U)\p_x(A(U))(\p_x\UU,U_{xx})
+2\p_x(J(U))\p_x(A(U))(\p_x\UU,U_x)
\nonumber
\\
&\quad
+
2J(U)\p_x^2(A(U))(\p_x\UU,U_x)
+J(U)\p_x(A(U))(\UU,U_{xxx})
\nonumber
\\
&\quad
+J(U)\p_x^3(A(U))(\UU,U_x)
\nonumber
\\
&\quad
+r_{1}(U,U_x,\UU,\p_x\UU,\p_x^2\UU)
+\mathcal{O}(|U|+|U_x|+|\UU|),
\label{eq:I_1}
\end{align}
where 
\begin{align}
r_{1}(U,U_x,\UU,\p_x\UU,\p_x^2\UU)
&=
A(U)(\mathcal{O}(|\p_x^2\UU|+|\p_x\UU|),|U_{xx}|+|U_x|)
\nonumber
\\&=\sum_k
\mathcal{O}(|\p_x^2\UU|+|\p_x\UU|)\nu_k(U).
\label{eq:r1}
\end{align} 
Thus, we can rewrite as 
\begin{align}
I(U)
&=\p_x^2\{ J(U)\p_x^2\UU\}
+2A(U)(J(U)\p_x^3\UU,U_x)
\nonumber
\\
&\quad
+G^{II}_1+G^{II}_2
+2G^{I}_1+2G^{I}_2+2G^{I}_3
+G^{I}_4+G^{I}_5
\nonumber
\\
&\quad
+\sum_k\mathcal{O}(|\p_x^2\UU|+|\p_x\UU|)\nu_k(U)
+\mathcal{O}(|U|+|U_x|+|\UU|), 
\label{eq:W_1}
\end{align}
where 
\begin{align}
&G^{II}_1
=J(U)\p_x(A(U))(\p_x^2\UU,U_x), 
\quad
G^{II}_2=\p_x(A(U))(J(U)\p_x^2\UU,U_x), 
\nonumber
\\
&G_1^{I}
=J(U)\p_x(A(U))(\p_x\UU,U_{xx}), 
\quad
G^{I}_2=\p_x(J(U))\p_x(A(U))(\p_x\UU,U_x), 
\nonumber
\\
&G^{I}_3
=
J(U)\p_x^2(A(U))(\p_x\UU,U_x), 
\quad
G^{I}_4=J(U)\p_x(A(U))(\UU,U_{xxx}), 
\nonumber
\\
&G_{5}^{I}
=J(U)\p_x^3(A(U))(\UU,U_x).
\nonumber
\end{align}
For $G_1^{II}$, we can express by 
\begin{align}
G^{II}_1
&=T(U)\p_x^2\UU,
\label{eq:5mA0}
\end{align} 
where  $T(U)$ has been defined 
by \eqref{eq:T(U)}. 
\\
For $G_2^{II}$, we use \eqref{eq:A1} with 
$Y_1=J(U)\p_x^2\UU$ and $Y_2=U_x$ to see 
\begin{align}
G_2^{II}
&=\sum_k(J(U)\p_x^2\UU,\p(D_k(U)P(U))U_x)\nu_k(U)
\nonumber
\\
&\quad
+\sum_k(\p_x(P(U))J(U)\p_x^2\UU,D_k(U)U_x)\nu_k(U)
\nonumber
\\&\quad
+\sum_k(J(U)\p_x^2\UU,D_k(U)U_x)D_k(U)U_x.
\nonumber
\\
&=
-\sum_k(\p_x^2\UU,J(U)D_k(U)U_x)D_k(U)U_x
+\sum_{k}\mathcal{O}(|\p_x^2\UU|)\nu_{k}(U).
\nonumber
\end{align}
Moreover, from $N(U)D_k(U)U_x
=\displaystyle\sum_{\ell}(D_k(U)U_x,\nu_{\ell}(U))\nu_{\ell}(U)$, 
it follows that
\begin{align}
G_2^{II}
&=
-\sum_k(\p_x^2\UU,J(U)D_k(U)U_x)P(U)D_k(U)U_x
+\sum_{k}\mathcal{O}(|\p_x^2\UU|)\nu_{k}(U).
\nonumber
\end{align}
Recalling $(T(U))^{\ast}$ has been given by \eqref{eq:Tstar}, 
we have
\begin{align}
G^{II}_2
&=
-(T(U))^{\ast}\p_x^2\UU
+\sum_{k}\mathcal{O}(|\p_x^2\UU|)\nu_{k}(U).
\label{eq:6mA}
\end{align}
Combining \eqref{eq:5mA0} and \eqref{eq:6mA}, 
we have 
\begin{align}
G_{1}^{II}+G_2^{II}
&=
(T(U)-(T(U))^{\ast})\p_x^2\UU
+\sum_k\mathcal{O}(|\p_x^2\UU|)\nu_k(U), 
\label{eq:7mA}
\end{align}
which combined with \eqref{eq:331} shows 
\begin{align}
G_{1}^{II}+G_2^{II}
&=
-R(J(U)U_x,U_x)\p_x^2\UU
+
S_{+}(U)\p_x^2\UU
+\sum_k\mathcal{O}(|\p_x^2\UU|)\nu_k(U).
\label{eq:7mAA}
\end{align}
Here, from \eqref{eq:8mmm}, it follows that  
\begin{align}
&-R(J(U)U_x,U_x)\p_x^2\UU
\nonumber
\\
&=
-\p_x
\left\{
R(J(U)U_x,U_x)\p_x\UU
\right\}
+2R(J(U)\UU,U_x)\p_x\UU
\nonumber
\\
&\quad
-A(U)(R(J(U)U_x,U_x)\p_x\UU, U_x)
+
\mathcal{O}(|U|+|U_x|+|\UU|).
\label{eq:531}
\end{align}
Substituting \eqref{eq:531} into \eqref{eq:7mAA}, 
and rewriting the second term of the right hand side of 
\eqref{eq:7mAA} in divergence form,  
we arrive at 
\begin{align}
G_{1}^{II}+G_2^{II}
&=
-\p_x
\left\{
R(J(U)U_x,U_x)\p_x\UU
\right\}
+2R(J(U)\UU,U_x)\p_x\UU
\nonumber
\\
&\quad 
+
\p_x\left\{
S_{+}(U)\p_x\UU
\right\}
-
\p_x\left(
S_{+}(U)
\right)
\p_x\UU
\nonumber
\\
&\quad 
+\sum_k\mathcal{O}(|\p_x^2\UU|+|\p_x\UU|)\nu_k(U)
+\mathcal{O}(|U|+|U_x|+|\UU|).
\label{eq:8mA}
\end{align}
For $2G^{I}_1$, from \eqref{eq:A2} with $Y_1=\p_x\UU$ 
and $Y_2=U_x$, 
we see
\begin{align}
2G^{I}_1
&=
2J(U)\sum_k(P(U)\p_x\UU,D_k(U)P(U)U_{xx})P(U)D_k(U)U_x.
\label{eq:9mA}
\end{align}
Next we compute $2G^I_2+2G^I_2$. 
For $2G^I_2$, 
from \eqref{eq:A3} with $Y_1=\p_x\UU$ 
and $Y_2=U_x$, 
we see 
\begin{align}
2G^I_2
&=
-2\sum_k(P(U)\p_x\UU,\p_x(D_k(U)P(U))U_x)J(U)D_k(U)U_x
\nonumber
\\
&\quad
-2\sum_k(\p_x(P(U))\p_x\UU,D_k(U)U_x)
J(U)D_k(U)U_x
\nonumber
\\
&\quad 
-2\sum_{k,\ell}(P(U)\p_x\UU,D_k(U)U_x)(D_k(U)U_x,\nu_{\ell}(U))
J(U)D_{\ell}(U)U_{x}
\nonumber
\\
&\quad
+\sum_{\ell}\mathcal{O}(|\p_x\UU|)\nu_{\ell}(U).
\label{eq:3te}
\end{align}
For $2G^I_3$, 
from \eqref{eq:A4} with $Y_1=\p_x\UU$ 
and $Y_2=U_x$, 
we see 
\begin{align}
2G_3^I
&=
4\sum_k(P(U)\p_x\UU,\p_x(D_k(U)P(U))U_x)J(U)D_k(U)U_x
\nonumber
\\
&\quad
+
4\sum_k(\p_x(P(U))\p_x\UU,D_k(U)U_x)J(U)D_k(U)U_x
\nonumber
\\
&\quad
+2\sum_k(P(U)\p_x\UU,D_k(U)U_x)J(U)\p_x\{D_k(U)U_x\}. 
\label{eq:4te}
\end{align}
Using \eqref{eq:kaehler2}, we see 
\begin{align}
J(U)\p_x\{D_k(U)U_x\}
&=\p_x\{J(U)D_k(U)U_x\}
-\p_x(J(U))D_k(U)U_x
\nonumber
\\
&=
\p_x\{J(U)D_k(U)U_x\}
+A(U)(J(U)D_k(U)U_x,U_x)
\nonumber
\\
&\quad
+\sum_{\ell}
(D_k(U)U_x,\nu_{\ell}(U))J(U)D_{\ell}(U)U_x.
\nonumber
\end{align}
Substituting this into \eqref{eq:4te}, 
and combining with \eqref{eq:3te}, 
we obtain 
\begin{align}
2G^I_2+2G^I_{3}
&=
(4-2)\sum_k(P(U)\p_x\UU,\p_x(D_k(U)P(U))U_x)J(U)D_k(U)U_x
\nonumber
\\
&\quad 
+2\sum_k(P(U)\p_x\UU,D_k(U)U_x)\p_x\{J(U)D_k(U)U_x\}
\nonumber
\\
&\quad 
+
2\sum_k(P(U)\p_x\UU,D_k(U)U_x)A(U)(J(U)D_k(U)U_x,U_x)
\nonumber
\\
&\quad
+(2-2)\sum_{k,\ell}
(P(U)\p_x\UU,D_k(U)U_x)(D_k(U)U_x,\nu_{\ell}(U))
J(U)D_{\ell}(U)U_{x}
\nonumber
\\
&\quad
+(4-2)\sum_k(\p_x(P(U))\p_x\UU,D_k(U)U_x)J(U)D_k(U)U_x
\nonumber
\\
&\quad
+\sum_{\ell}\mathcal{O}(|\p_x\UU|)\nu_{\ell}(U).
\nonumber
\end{align}
Since $\p_x(D_k(U)P(U))U_x=\p_x\{D_k(U)U_x\}-D_k(U)P(U)U_{xx}$ 
follows from 
$P(U)U_x=U_x$, we obtain 
\begin{align}
2G^I_2+2G^I_{3}
&=
-2\sum_k(P(U)\p_x\UU,D_k(U)P(U)U_{xx})J(U)D_k(U)U_x
\nonumber
\\
&\quad
+2\sum_k(P(U)\p_x\UU,\p_x\{D_k(U)U_x\})J(U)D_k(U)U_x
\nonumber
\\
&\quad 
+2\sum_k(P(U)\p_x\UU,D_k(U)U_x)\p_x\{J(U)D_k(U)U_x\}
\nonumber
\\
&\quad
+2\sum_k(\p_x(P(U))\p_x\UU,D_k(U)U_x)J(U)D_k(U)U_x
\nonumber
\\
&\quad
+\sum_{\ell}\mathcal{O}(|\p_x\UU|)\nu_{\ell}(U). 
\nonumber
\end{align}
Furthermore, recalling $T(U)=J(U)\p_x(A(U))(\cdot,U_x)$ by definition 
and using the 
expression \eqref{eq:TU1} for $\p_x(T(U))$, we can write
\begin{align}
2G^I_2+2G^I_{3}
&=
-2\sum_k(P(U)\p_x\UU,D_k(U)P(U)U_{xx})J(U)D_k(U)U_x
\nonumber
\\
&\quad
+2\p_x(T(U))\p_x\UU
+\sum_{\ell}\mathcal{O}(|\p_x\UU|)\nu_{\ell}(U).
\nonumber
\end{align} 
Observing \eqref{eq:9mA}, 
the first term of the right hand of the above 
is just  $-2G^{I}_1$.
Thus,  we get 
\begin{align}
2G^I_1+2G^I_2+2G^I_{3}
&=
2\p_x(T(U))\p_x\UU
+\sum_{\ell}\mathcal{O}(|\p_x\UU|)\nu_{\ell}(U).
\label{eq:66te}
\end{align} 
For the first term of the right hand side, 
we now apply \eqref{eq:ppp3} in Proposition~\ref{proposition:2TU} 
with $Y=\p_x\UU$. 
In the computation, note first that 
\begin{align}
&J(U)R(U_x,\p_x(N(U))\p_x\UU)U_x
\nonumber
\\
&=
J(U)R(U_x,P(U)\p_x(N(U))\p_x\UU)U_x
\qquad (\because \eqref{eq:RP})
\nonumber
\\
&=
J(U)\sum_{\ell}(\p_x\UU,\nu_{\ell}(U))R(U_x,D_{\ell}(U_x)U_x)U_x
\qquad 
(\because \eqref{eq:PN})
\nonumber
\\
&=\mathcal{O}(|U|+|U_x|+|\UU|)
\nonumber
\end{align}
holds, where in the last equality we use  
$(\p_x\UU,\nu_k(U))=-(\UU,D_k(U)U_x)$ which follows from 
$\UU\perp \nu_k(U)$. 
Note second that 
$$A(U)(J(U)R(U_x,\p_x\UU)U_x,U_x)
=\sum_{\ell}\mathcal{O}(|\p_x\UU|)\nu_{\ell}(U)$$
holds. Noting them, we apply 
\eqref{eq:ppp3} with $Y=\p_x\UU$, and then 
obtain 
\begin{align}
\p_x(T(U))\p_x\UU
&=
-R(J(U)\UU,U_x)\p_x\UU
-\frac{1}{2}(B_2(U)+B_3(U))\p_x\UU
\nonumber
\\
&\quad
+\sum_k\p_x\left(
(U_x,D_k(U)U_x)J(U)D_k(U)P(U)
\right)\p_x\UU
\nonumber
\\
&\quad
+\sum_{\ell}\mathcal{O}(|\p_x\UU|)\nu_{\ell}(U)
+\mathcal{O}(|U|+|U_x|+|\UU|).
\nonumber
\end{align}
Therefore, we obtain as follows:
\begin{align}
2G^I_1+2G^I_2+2G^I_{3}
&=
-2R(J(U)\UU,U_x)\p_x\UU
-(B_2(U)+B_3(U))\p_x\UU
\nonumber
\\
&\quad
+2\sum_k\p_x\left(
(U_x,D_k(U)U_x)J(U)D_k(U)P(U)
\right)\p_x\UU
\nonumber
\\
&\quad
+\sum_{\ell}\mathcal{O}(|\p_x\UU|)\nu_{\ell}(U)
+\mathcal{O}(|U|+|U_x|+|\UU|).
\label{eq:67te}
\end{align}
For $G^{I}_4$, 
from \eqref{eq:A2} with $Y_1=\UU=P(U)\UU$ 
and $Y_2=U_{xxx}$,  
we see
$$
G^I_4
=\sum_k
(\UU,D_k(U)P(U)U_{xxx})J(U)D_k(U)U_x.
$$
Since $U_{xxx}=\p_x\UU+\mathcal{O}(|U|+|U_x|+|\UU|)$,  
we deduce 
\begin{align}
G^I_4
&=\sum_k
(\UU, D_k(U)P(U)\p_x\UU)J(U)D_k(U)U_x
+\mathcal{O}(|U|+|U_x|+|\UU|)
\nonumber
\\
&=
J(U)\sum_k(\UU, D_k(U)P(U)\p_x\UU)P(U)D_k(U)U_x
+\mathcal{O}(|U|+|U_x|+|\UU|).
\label{eq:7te}
\end{align}
Applying \eqref{eq:embcurvature} 
for the first term of the right hand side 
with 
$Y_1=U_x$, $Y_2=\p_x\UU$, 
$Y_3=\UU$, and using 
$P(U)\p_x\UU=\p_x\UU+\mathcal{O}(|U|+|U_x|+|\UU|)$, 
we deduce 
\begin{align}
G_4^{I}
&=
J(U)R(U_x,\p_x\UU)\UU
+
J(U)\sum_k
(\UU, D_k(U)U_x)P(U)D_k(U)P(U)\p_x\UU
\nonumber
\\
&\quad
+\mathcal{O}(|U|+|U_x|+|\UU|) 
\nonumber
\\
&=
-R(\p_x\UU,U_x)J(U)\UU
+
\sum_k
(\UU, D_k(U)U_x)J(U)D_k(U)\p_x\UU
\nonumber
\\
&\quad
+\mathcal{O}(|U|+|U_x|+|\UU|). 
\nonumber
\end{align}
Furthermore, applying \eqref{eq:ppp2} with 
$Y=\p_x\UU$, we arrive at
\begin{align}
G^I_4 &=
-\frac{1}{2}R(J(U)\UU,U_x)\p_x\UU
+\left(\frac{1}{2}B_1(U)-\frac{1}{4}B_2(U)
-\frac{1}{4}B_3(U)\right)\p_x\UU
\nonumber
\\
&\quad
+
\sum_k
(\UU, D_k(U)U_x)J(U)D_k(U)\p_x\UU
+\mathcal{O}(|U|+|U_x|+|\UU|). 
\label{eq:12mA}
\end{align}
For $G^I_5=J(U)\p_x^3(A(U))(\UU,U_x)$, 
starting from taking the partial derivative of 
\eqref{eq:A1} twice with respect to $x$,  
we see
\begin{align}
&\p_x^3(A(U))(\UU,U_x)
\nonumber
\\
&=
\sum_k(P(U)\UU,\p_x^3(D_k(U)P(U))U_x)\nu_k(U)
+\sum_k(\p_x^3(P(U))\UU,D_k(U)P(U)U_x)\nu_k(U)
\nonumber
\\
&\quad
+\sum_k(P(U)\UU,D_k(U)P(U)U_x)D_k(U)U_{xxx}
+\mathcal{O}(|U|+|U_x|+|U_{xx}|). 
\nonumber
\end{align}
Since $J(U)\nu_k(U)=0$ and
$U_{xxx}=\p_x\UU
+\mathcal{O}(|U|+|U_x|+|\UU|)$, we have
\begin{align}
G^{I}_5
&=
\sum_k(\UU,D_k(U)U_x)J(U)D_k(U)\p_x\UU
+\mathcal{O}(|U|+|U_x|+|\UU|).
\label{eq:13mA}
\end{align}
Gathering the information 
\eqref{eq:8mA}, 
\eqref{eq:67te},
\eqref{eq:12mA} and
\eqref{eq:13mA}, 
we obtain 
\begin{align}
&G^{II}_1+G^{II}_2
+2G^I_1+2G^I_2+2G^{I}_3
+G^I_4+G^I_5
\nonumber
\\
&=
-\p_x
\left\{
R(J(U)U_x,U_x)\p_x\UU
\right\}
+\p_x
\left\{
S_{+}(U)\p_x\UU
\right\} 
\nonumber
\\
&\quad 
+
\p_x
\left(
-S_{+}(U)+2\sum_k
(U_x,D_k(U)U_x)P(U)J(U)D_k(U)P(U)
\right)
\p_x\UU
\nonumber
\\
&\quad
-\frac{1}{2}R(J(U)\UU,U_x)\p_x\UU
+\left(\frac{1}{2}B_1(U)-\frac{5}{4}B_2(U)
-\frac{5}{4}B_3(U)\right)\p_x\UU
\nonumber
\\
&\quad
+
2\sum_k
(\UU, D_k(U)U_x)J(U)D_k(U)\p_x\UU
\nonumber
\\
&\quad
+\sum_k\mathcal{O}(|\p_x^2\UU|+|\p_x\UU|)\nu_k(U)
+\mathcal{O}(|U|+|U_x|+|\UU|)
\nonumber
\\
&=
-\p_x
\left\{
R(J(U)U_x,U_x)\p_x\UU
\right\}
+\p_x
\left\{
S_{+}(U)\p_x\UU
\right\} 
+
\p_x(S_{-}(U))
\p_x\UU
\nonumber
\\
&\quad
-\frac{1}{2}R(J(U)\UU,U_x)\p_x\UU
+\left(\frac{1}{2}B_1(U)-\frac{5}{4}B_2(U)
-\frac{5}{4}B_3(U)\right)\p_x\UU
\nonumber
\\
&\quad
+
2\sum_k
(\UU, D_k(U)U_x)J(U)D_k(U)\p_x\UU
\nonumber
\\
&\quad
+\sum_k\mathcal{O}(|\p_x^2\UU|+|\p_x\UU|)\nu_k(U)
+\mathcal{O}(|U|+|U_x|+|\UU|).
\label{eq:y5}
\end{align}
Substituting 
\eqref{eq:y5} into \eqref{eq:W_1}, 
we get the expression of $I(U)=dw(\nabla_x^2J_u\nabla_x^3u_x)$. 
Consequently, 
substituting it multiplied by $a$ into \eqref{eq:1UU}, and 
combining with \eqref{eq:r}, 
we derive the desired \eqref{eq:UU} with 
\eqref{eq:UU2}. 
\end{proof}
\subsection{Classical Energy Estimate for $\WW$ in $L^2$}
\label{subsection:energy}
We derive the classical energy estimate for $\WW$ in $L^2$. 
In what follows, 
the inner product and the norm in $L^2$
for $\RR^d$-valued functions on $\TT$
will be denoted by
$\lr{\cdot}$ 
and 
$\|\cdot\|_{L^2}$ respectively. 
That is, 
for $\phi, \psi:\TT\to \RR^d$, 
$\lr{\phi,\psi}$ and $\|\phi\|_{L^2}$ 
is given by 
$
\lr{\phi,\psi}
=
\int_{\TT}
(\phi(x),\psi(x))
\,dx$ 
and  
$\|\phi\|_{L^2}
=\sqrt{\lr{\phi,\phi}}
$
respectively. 
In the same way, the standard Sobolev norm in $H^k(\TT;\RR^d)$ 
with integer $k\geqslant 1$ is denoted by $\|\cdot\|_{H^k}$.
Moreover, 
we recall  the following properties hold:
\begin{align}
&\UU,\VV, \WW\in L^{\infty}(-T,T;H^4(\TT;\RR^d))
\cap C([-T,T];H^4(\TT;\RR^d)), 
\nonumber 
\\
&\p_x^{\ell}U_x, \p_x^{\ell}V_x\in 
L^{\infty}(-T,T; C(\TT;\RR^d))
\quad  
(\ell=0,1,\ldots,4), 
\nonumber
\\
&\p_x^{\ell}\UU, \p_x^{\ell}\VV\in 
L^{\infty}(-T,T; C(\TT;\RR^d))
\quad (\ell=0,1,\ldots,3),
\nonumber
\end{align} 
since $k\geqslant 5$ is imposed. 
\par
Noting $\VV$ also satisfies \eqref{eq:UU} 
replacing $(U,U_x,\UU)$ with $(V,V_x,\VV)$, 
we take the difference of $\p_t\UU$ and $\p_t\VV$,   
which yields 
\begin{align}
\p_t\WW
&=L(U)\WW
+a(J(U)-J(V))\p_x^4\VV
\nonumber
\\
&\quad 
+\sum_k\mathcal{O}(|\p_x^2\WW|+|\p_x\WW|)\nu_k(U)
+\mathcal{O}(|Z|+|Z_x|+|\WW|),
\label{eq:WW}
\end{align} 
for almost every $t\in (-T,T)$, 
where $L(U)$ is the partial differential operator given 
in Proposition~\ref{proposition:UU}. 
Since $\WW\in L^{\infty}(-T,T;H^4(\TT;\RR^d))$ as commented 
above, we see  $\p_t\WW\in L^{\infty}(-T,T;L^2(\TT;\RR^d))$ 
holds.  
Hence, following the argument in, e.g.,  
\cite[Lemma~3.2]{Temam}
or
\cite[Theorem~7.2]{Robinson}, 
we can show  
\begin{align}
&\frac{1}{2}
\frac{d}{dt} 
\|\WW(t)\|_{L^2}^2
=
\lr{\p_t\WW(t),\WW(t)}
\label{eq:4291add}
\end{align}
holds for almost every $t\in (-T,T)$.
Using \eqref{eq:WW} and \eqref{eq:UU2},  
we have 
\begin{align}
&
\lr{\p_t\WW,\WW}
\nonumber
\\
&=
a\lr{
\p_x^2\{J(U)\p_x^2\WW\}, \WW
}
+2a
\lr{
A(U)(J(U)\p_x^3\WW,U_x),\WW
}
\nonumber
\\
&\quad
+\lambda
\lr{
\p_x\left\{
J(U)\p_x\WW
\right\}
,\WW} 
+(-a+b)\lr{R(\p_x^2\WW,J(U)U_x)U_x
,\WW}
\nonumber
\\
&\quad
+(-a+b+c)
\lr{\p_x\{R(J(U)U_x,U_x)\p_x\WW\}, \WW
}
\nonumber
\\
&\quad
+
\left(-\frac{a}{2}-\frac{b}{2}+3c\right)
\lr{
R(J(U)\UU,U_x)\p_x\WW,\WW
}
+\sum_{i=1}^3
c_i\lr{B_i(U)\p_x\WW,\WW}
\nonumber
\\
&\quad
+
a\lr{\p_x\{S_{+}(U)\p_x\WW\},\WW}
+a\lr{
\p_x(S_{-}(U))\p_x\WW,\WW
}
\nonumber
\\
&\quad
+2a\lr{
\sum_k(\UU,D_k(U)U_x)J(U)D_k(U)\p_x\WW,\WW
}
\nonumber
\\
&\quad 
+a\lr{(J(U)-J(V))\p_x^4\VV,\WW}
\nonumber
\\
&\quad
+
\lr{
\sum_k
\mathcal{O}
(|\p_x^2\WW|+|\p_x\WW|)\nu_k(U), 
\WW
}
+
\lr{
\mathcal{O}
(|Z|+|Z_x|+|\WW|), 
\WW
}
\label{eq:EEn}
\end{align}
for almost every $t\in (-T,T)$. 
We furthermore compute the right hand side of the above term by term. 
Before that,  
we observe some key properties related to $\nu_k(U)$ and $D_k(U)$ 
for $k=2n+1\ldots,d$.  
\begin{proposition}
\label{proposition:nu}
Under the same setting as above, 
the following properties hold. 
\begin{align}
(\nu_k(U), \WW)&=
-(\nu_k(U)-\nu_k(V), \VV)
=\mathcal{O}(|Z|), 
\label{eq:key1}
\\
(\nu_k(U),\p_x\WW)
&=
-(D_k(U)U_x,\WW)
-(D_k(U)Z_x,\VV)
+
\mathcal{O}(|Z|). 
\label{eq:key2}
\end{align}
\end{proposition}
The proof is given in \cite[Lemma~3.1]{onodera4}, and thus we omit the 
detail.
In view of \eqref{eq:key1}, 
we find
that the terms 
including $\p_x^2\WW$ or $\p_x\WW$ 
can be handled as a harmless term if they are expressed as a 
linear combination of $\nu_{2n+1}(U), \ldots, \nu_d(U)$.  
See, e.g., the argument to show \eqref{eq:18e}.
Here and hereafter, 
various positive constants depending on 
$\|U_x\|_{L^{\infty}(-T,T;H^5)}$ and 
$\|V_x\|_{L^{\infty}(-T,T;H^5)}$
will be denoted by the same $C$
without any comments. 
Besides, we define $D(t)$ so that the square is given by 
$$D(t)^2=\|Z(t)\|_{L^2}^2
+\|Z_x(t)\|_{L^2}^2
+\|\WW(t)\|_{L^2}^2.
$$ 
\par 
Going back to \eqref{eq:EEn}, 
we use the integration by parts and the skew-symmetry of 
$J(U)$ to see
\begin{align}
&a\lr{
\p_x^2\{J(U)\p_x^2\WW\}, \WW
}
=
a\lr{
J(U)\p_x^2\WW, \p_x^2\WW
}=0,
\label{eq:y6}
\\
&\lambda\,
\lr{
\p_x\left\{
J(U)\p_x\WW
\right\}
,\WW
}
=
-\lambda\,
\lr{
J(U)\p_x\WW
,\p_x\WW
}
=0. 
\label{eq:10tte}
\end{align} 
By the Cauchy-Schwartz inequality, 
we deduce  
\begin{align}
\lr{
\mathcal{O}
(|Z|+|Z_x|+|\WW|), 
\WW
}
&\leqslant 
\|\mathcal{O}
(|Z|+|Z_x|+|\WW|)\|_{L^2}
\|\WW\|_{L^2} 
\leqslant 
C\,D(t)^2. 
\label{eq:11tte}
\end{align}
By \eqref{eq:key1}, we see $(\nu_k(U),\WW)=\mathcal{O}(|Z|)$.
Thus, integrating by parts, 
we obtain
\begin{align}
\lr{
\sum_k
\mathcal{O}
(|\p_x^2\WW|+|\p_x\WW|)\nu_k(U), 
\WW
}
&\leqslant 
C\,D(t)^2.
\label{eq:18e}
\end{align}
Integrating by parts, and then using 
\eqref{eq:R1}-\eqref{eq:R2}, 
we have   
\begin{align}
&\lr{\p_x\left\{
R(J(U)U_x,U_x)\p_x\WW
\right\}, \WW}
=
-\lr{R(J(U)U_x,U_x)\p_x\WW, \p_x\WW}
=0.
\label{eq:10mm}
\end{align} 
Noting the  symmetry of $B_i(U)$ discussed 
in  Proposition~\ref{proposition:B_isym},
and integrating by parts, we deduce   
\begin{align}
\lr{B_i(U)\p_x\WW, \WW}
&=
-\frac{1}{2}\lr{\p_x(B_i(U))\WW,\WW}
\leqslant 
C\,D(t)^2. 
\label{eq:11mm}
\end{align}
Recall that $S_{+}(U)$ and $S_{-}(U)$ defined in 
Definition~\ref{definition:spm}  or 
\eqref{eq:331}-\eqref{eq:431}
are respectively skew-symmetric and symmetric, that is, 
$(S_{\pm}(U)Y_1,Y_2)=\mp (Y_1,S_{\pm}(U)Y_2)$ 
for any $Y_1,Y_2:[-T,T]\times \TT\to \RR^d$. 
Using the skew-symmetry of $S_{+}(U)$  
and the symmetry of $\p_x(S_{-}(U))$, 
we have 
\begin{align}
&\lr{\p_x\{S_{+}(U)\p_x\WW\}, \WW}
=
-\lr{
S_{+}(U)\p_x\WW, \p_x\WW
}
=0, 
\label{eq:y7}
\\
&
\lr{\p_x(S_{-}(U))\p_x\WW,\WW}
\leqslant 
C\,D(t)^2.
\label{eq:y8}
\end{align}
By the Cauchy-Schwartz inequality, 
the Sobolev embedding 
$H^1(\TT;\RR^d)\subset C(\TT;\RR^d)$, 
and $\VV\in L^{\infty}(-T,T;H^4(\TT;\RR^d))$, 
we deduce
\begin{align}
\lr{(J(U)-J(V)\p_x^4\VV,\WW}
&\leqslant 
\|(J(U)-J(V))\p_x^4\VV\|_{L^2} \|\WW\|_{L^2}
\nonumber
\\&\leqslant 
C\|Z\|_{H^1}\|\p_x^4\VV\|_{L^2}\|\WW\|_{L^2}
\nonumber
\\
&\leqslant 
C\,D(t)^2.
\label{eq:423add}
\end{align}
We next set $E:=\lr{A(U)(J(U)\p_x^3\WW,U_x),\WW}$.  
Although the third order term $\p_x^3\WW$ is seemingly 
included in $E$, we will see that  
only a first-order loss of derivative occurs  
thanks to the property of $A(U)(\cdot,\cdot)$ 
as demonstrated below. 
Noting 
$A(U)(\mathcal{O}(|\p_x^2\WW|), |U_x|+|U_{xx}|)
=\sum_k\mathcal{O}(|\p_x^2\WW|)\nu_k(U)$ 
and \eqref{eq:key1}, 
we use the integration by parts to obtain 
\begin{align}
2E^{III}
&\leqslant
-2\lr{
\p_x(A(U))(J(U)\p_x^2\WW, U_x), \WW
}
\nonumber
\\
&\quad
-2\lr{
A(U)(J(U)\p_x^2\WW, \p_x\WW)
}
+C\,D(t)^2.
\label{eq:y1}
\end{align}
By  \eqref{eq:A1} and \eqref{eq:key1}, we see 
\begin{align}
&(\p_x(A(U))(J(U)\p_x^2\WW, U_x), \WW)
\nonumber
\\
&=
\sum_k
(J(U)\p_x^2\WW,\p_x(D_k(U)P(U))U_x)(\nu_k(U),\WW)
\nonumber
\\
&\quad
+
\sum_k
(J(U)\p_x^2\WW,D_k(U)U_x)(D_k(U)U_x,\WW)
\nonumber
\\
&\quad
+\sum_k(
\p_x(P(U))J(U)\p_x^2\WW,D_k(U)U_x)(\nu_k(U),\WW)
\nonumber
\\
&=
\sum_k
(J(U)\p_x^2\WW,D_k(U)U_x)(D_k(U)U_x,\WW)
\nonumber
\\
&\quad
+
\sum_k
(J(U)\p_x^2\WW,\p_x(D_k(U)P(U))U_x)\mathcal{O}(|Z|)
\nonumber
\\
&\quad
+\sum_k(
\p_x(P(U)J(U)\p_x^2\WW,D_k(U)U_x)\mathcal{O}(|Z|).
\label{eq:y2}
\end{align}
By \eqref{eq:key2}, we see
\begin{align}
(A(U)(J(U)\p_x^2\WW,U_x),\p_x\WW)
&=
\sum_k(J(U)\p_x^2\WW,D_k(U)U_x)(\nu_k(U),\p_x\WW)
\nonumber
\\
&=
-\sum_k(J(U)\p_x^2\WW,D_k(U)U_x)(D_k(U)U_x,\WW)
\nonumber
\\
&\quad
-\sum_k(J(U)\p_x^2\WW,D_k(U)U_x)(D_k(U)Z_x,\VV)
\nonumber
\\
&\quad
-\sum_k(J(U)\p_x^2\WW,D_k(U)U_x)\mathcal{O}(|Z|).
\label{eq:y3}
\end{align}
Combining \eqref{eq:y2} and \eqref{eq:y3}, 
and using the integration by parts,  
we have 
\begin{align}
2E
&\leqslant 
2\lr{\sum_k(J(U)\p_x^2\WW,D_k(U)U_x)D_k(U)Z_x,\VV}
+C\,D(t)^2
\nonumber
\\
&\leqslant
2\lr{\sum_k(J(U)\WW,D_k(U)U_x)D_k(U)Z_{xxx},\VV}
+C\,D(t)^2.
\label{eq:y4}
\end{align}
Furthermore, since 
$(D_k(U)Z_{xxx},\VV)=(D_k(U)Z_{xxx},-\WW+\UU)=
(D_k(U)Z_{xxx},\UU)+\mathcal{O}(|\WW|)$ and 
$Z_{xxx}=P(U)\p_x\WW+\mathcal{O}(|Z|+|Z_x|+|\WW|)$, 
we see
\begin{align}
(D_k(U)Z_{xxx},\VV)
&=(D_k(U)P(U)\p_x\WW,\UU)+\mathcal{O}(|Z|+|Z_x|+|\WW|).
\nonumber
\end{align}
Substituting this into \eqref{eq:y4}, we obtain
\begin{align}
2\,E
&\leqslant 
2\lr{
\sum_{k}
\left(J(U)\WW,
D_{k}(U)U_x
\right)
D_k(U)P(U)\p_x\WW
, \UU
}
+C\,D(t)^2.
\label{eq:skip}
\end{align}
Furthermore,  using $N(U)D_k(U)P(U)\p_x\WW \perp \UU$, 
we have 
\begin{align}
2\,E
&\leqslant 
2\lr{
\sum_{k}
\left(J(U)\WW,
D_{k}(U)U_x
\right)
P(U)D_k(U)P(U)\p_x\WW
, \UU
}
+C\,D(t)^2.
\nonumber
\end{align}
Then, applying 
\eqref{eq:embcurvature} 
with 
$Y_1=\p_x\WW$, 
$Y_2=U_x$
$Y_3=J(U)\WW$, 
we obtain 
\begin{align}
2\,E
&\leqslant 
2\lr{R(\p_x\WW,U_x)J(U)\WW,\UU}
\nonumber
\\
&\quad
+
2\lr{
\sum_k
\left(J(U)\WW,
D_{k}(U) P(U)\p_x\WW
\right)
P(U)D_k(U)U_x
, \UU
} 
+C\,D(t)^2.
\end{align}
Using \eqref{eq:R2} and \eqref{eq:R6} for the first term, 
and 
$N(U)\p_x\WW
=
\mathcal{O}
\left(
|Z|+|Z_x|+|\WW|
\right)$
and 
$(P(U)D_k(U)U_x,\UU)
=(D_k(U)U_x,P(U)\UU)=(D_k(U)U_x, \UU)$
for the second term, 
we deduce 
\begin{align}
2\,E
&\leqslant 
2\lr{R(\p_x\WW,U_x)J(U)\UU, \WW}
\nonumber
\\
&\quad
+
2\lr{
\sum_k
(P(U)D_k(U)U_x,\UU)
D_k(U)P(U)\p_x\WW
, J(U)\WW
} 
+C\,D(t)^2
\nonumber
\\
&\leqslant 
2\lr{R(\p_x\WW,U_x)J(U)\UU, \WW}
\nonumber
\\
&\quad
-
2\lr{
\sum_k
(\UU, D_k(U)U_x)
J(U)D_k(U)\p_x\WW
, \WW
} 
+C\,D(t)^2.
\nonumber
\end{align}
Applying \eqref{eq:ppp2} 
in Proposition~\ref{proposition:Rsym}
with 
$Y=\p_x\WW$, 
and noting 
$\lr{B_i(U)\p_x\WW,\WW}\leqslant C\, D(t)^2$ 
follows from the symmetry of $B_i(U)$ 
as we obtain \eqref{eq:11mm}, 
we obtain 
\begin{align}
&2a\lr{A(U)(J(U)\p_x^3\WW,U_x),\WW}
(=2a\,E)
\nonumber
\\
&\leqslant 
a\lr{R(J(U)\UU,U_x) \p_x\WW, \WW}
\nonumber
\\
&\quad
-2a
\lr{
\sum_k
(\UU, D_k(U)U_x)
J(U)D_k(U)\p_x\WW
, \WW
} 
+C\,D(t)^2.
\label{eq:3mo}
\end{align}
\par 
Finally, by substituting 
\eqref{eq:y6}--\eqref{eq:423add}, 
and \eqref{eq:3mo} 
into \eqref{eq:EEn} combined with \eqref{eq:4291add},
we conclude 
\begin{align}
\frac{1}{2}\frac{d}{dt}
\|\WW\|_{L^2}^2
&\leqslant
(-a+b)\lr{\,R(\p_x^2\WW,J(U)U_x)U_x,\WW}
\nonumber
\\
&\quad 
+\left(a-\frac{a}{2}-\frac{b}{2}+3c\right) 
\lr{R(J(U)\UU,U_x)\p_x\WW,\WW}
\nonumber
\\
&\quad
+(2a-2a)
\lr{
\sum_k
(\UU, D_k(U)U_x)
J(U)D_k(U)\p_x\WW
, \WW
} 
+C\,D(t)^2
\nonumber
\\
&=
(-a+b)\lr{\,R(\p_x^2\WW,J(U)U_x)U_x,\WW}
\nonumber
\\
&\quad 
+\left(\frac{a}{2}-\frac{b}{2}+3c\right) 
\lr{R(J(U)\UU,U_x)\p_x\WW,\WW}
+C\,D(t)^2
\label{eq:momo}
\end{align}
holds for almost every $t\in (-T,T)$. 
This combined with the integration by parts 
implies only the estimate of the form 
$\frac{1}{2}\frac{d}{dt}\|\WW\|_{L^2}^2
\leqslant C\,\|\WW\|_{H^1}^2$, 
which is still unsatisfactory. 
%
\section{Proof of Theorem~\ref{theorem:uniqueness}}
\label{section:proof}
In this section,  we complete the proof of
Theorem~\ref{theorem:uniqueness}. 
\begin{proof}[Proof of Theorem~\ref{theorem:uniqueness}]
Suppose 
$u$ and $v$ are solutions to 
\eqref{eq:pde}-\eqref{eq:data} 
in Theorem~\ref{theorem:existence}. 
We shall show $u=v$. 
To do this, 
we fix $w$ as an isometric embedding of $(N,J,h)$ into an 
Euclidean space $\RR^d$, 
and  set 
$U=w{\circ}u$, 
$V=w{\circ}v$, 
$Z=U-V$, 
$\UU=dw_u(\nabla_xu_x)$, 
$\VV=dw_v(\nabla_xv_x)$, 
and $\WW=\UU-\VV$. 
The setting is completely the same as that in the previous section. 
Hence, we do not repeat but 
it is recommendable to recall the comments  
stated at the beginning of Section~\ref{section:cl}. 
\par 
For this purpose,  
as $w$ is injective, 
it suffices to show $Z=0$. 
We consider the estimate 
for the following modified energy 
\begin{align}
\widetilde{D}(t)^2
&:=
\|Z(t)\|_{L^2}^2
+\|Z_x(t)\|_{L^2}^2
+\|\TW(t)\|_{L^2}^2, 
\label{eq:menergy2}
\\
\widetilde{\WW}
&:=
\WW+\Lambda,
\label{eq:sag1}
\end{align}
where    
\begin{align}
\Lambda
&:=
-\frac{e_1}{2a}R(Z,U_x)U_x
+
\frac{e_2}{8a} R(J(U)U_x,U_x)J(U)Z, 
\label{eq:sag2}
\end{align}
and $e_1,e_2\in \RR$ are constants which 
will be taken later.
Since $u$ and $v$ satisfy the same initial value, 
$\widetilde{D}(0)=0$ holds. 
We shall show that  
there exists a positive constant $C$ such that
\begin{equation}
\frac{1}{2}
\frac{d}{dt}\widetilde{D}(t)^2
\leqslant C\, 
\widetilde{D}(t)^2
\label{eq:De}
\end{equation}
for almost every $t\in (-T,T)$. 
If it actually holds, then   
\eqref{eq:De} together with $\widetilde{D}(0)=0$ 
shows 
$$
\widetilde{D}(t)^2
=
\widetilde{D}(0)^2
+
\int_0^t\frac{d}{dt}\widetilde{D}(s)^2\,ds
\leqslant 
2C\int_0^t\widetilde{D}(s)^2\,ds
$$
for all $t\in (-T,T)$. 
Since $\widetilde{D}^2(t)$ is a continuous 
real valued function on $[-T,T]$, the inequality shows 
$\widetilde{D}(t)\equiv 0$, which yields $Z(t)=0$ for all $t\in [-T,T]$.
\par
In the same way as we obtain \eqref{eq:momo},  
it is now not difficult to obtain the following estimate 
for almost every $t\in (-T,T)$, 
permitting the loss of derivatives of order one:  
\begin{equation}
\frac{1}{2}\frac{d}{dt}
\left\{
\|Z(t)\|_{L^2}^2
+
\|Z_x(t)\|_{L^2}^2
\right\}
\leqslant 
C\, \widetilde{D}(t)^2.  
\label{eq:notmainineq}
\end{equation}
Hence, 
we hereafter concentrate on deriving 
the equation satisfied by 
$\TW 
$
and the estimate for 
$
\|\TW(t)\|_{L^2}^2
$. 
Observing $\Lambda=\mathcal{O}(|Z|)$, 
we see 
$\TW=\WW+\mathcal{O}(|Z|)$, 
$\p_x\TW=\p_x\WW+\mathcal{O}(|Z|+|Z_x|)$, 
and 
$\p_x^2\TW=\p_x^2\WW+\mathcal{O}(|Z|+|Z_x|+|\WW|)$, 
which will be often used without comments. 
Particularly,  
we will sometimes write $f\equiv g$ 
for $f,g:(-T,T)\times \TT\to \RR^d$ 
if $f-g=\mathcal{O}(|Z|+|Z_x|+|\WW|)$ 
holds almost everywhere. 
Moreover, we will sometimes use the property \eqref{eq:RP} 
without comments. 
\par
We start the computation of  
$\p_t\TW=\p_t\WW+\p_t\Lambda$. 
Using \eqref{eq:R7} and noting 
$\p_tU_x=\p_xU_t\in L^{\infty}(-T;T;C(\TT;\RR^d))$, 
we see     
\begin{align}
\p_t\Lambda
&=
-\frac{e_1}{2a}
R(\p_t\{P(U)Z\},U_x)U_x
+
\frac{e_2}{8a}R(J(U)U_x,U_x)J(U)\p_t\{P(U)Z\}
+\mathcal{O}(|Z|)
\nonumber
\\
&=
-\frac{e_1}{2a}
R(P(U)Z_t,U_x)U_x
+
\frac{e_2}{8a}R(J(U)U_x,U_x)J(U)P(U)Z_t
+\mathcal{O}(|Z|)
\label{eq:aya1}
\end{align} 
holds for almost every $t$.
Recalling \eqref{eq:U_t} and using \eqref{eq:key2} and \eqref{eq:kaehler2}, we see  
\begin{align}
Z_t
&\equiv
a\p_x\{
J(U)\p_x\WW
\}+
a\,A(U)(J(U)\p_x\WW,U_x)
\nonumber
\\
&=
a\,J(U)\p_x^2\WW
+a\p_x(J(U))\p_x\WW+
a\,A(U)(J(U)\p_x\WW,U_x)
\nonumber
\\
&\equiv 
a\,J(U)\p_x^2\WW.
\nonumber
\end{align}
Moreover, since 
\begin{align}
\WW
&=
U_{xx}+A(U)(U_x,U_x)-V_{xx}-A(V)(V_x,V_x)
\nonumber
\\
&=Z_{xx}+A(U)(U_x,Z_x)+
A(U)(Z_x,U_x)+(A(U)-A(V))(V_x,V_x)
\label{eq:130}
\end{align}
and $J(U)A(U)(\cdot,\cdot)=0$, 
we see
$$
J(U)\p_x^2\WW
\equiv 
J(U)\left(
\p_x^2Z_{xx}
+A(U)(U_x,\p_xZ_{xx})
+A(U)(\p_xZ_{xx},U_x)
\right)
=J(U)\p_x^2Z_{xx}.
$$
This implies
\begin{align}
Z_t&\equiv 
a J(U)\p_x^2Z_{xx}.
\label{eq:aya3}
\end{align}
From this and \eqref{eq:aya1}, we have
\begin{align}
\p_t\Lambda
&\equiv
-\frac{e_1}{2}
R(J(U)\p_x^2Z_{xx},U_x)U_x
-
\frac{e_2}{8}R(J(U)U_x,U_x)\p_x^2Z_{xx}.
\label{eq:2momo}
\end{align}
On the other hand,   
recalling \eqref{eq:WW}, 
we can write 
\begin{align}
\p_t\WW
&=
L(U)\WW+a(J(U)-J(V))\p_x^4\VV
\nonumber
\\
&\quad
+\sum_k\mathcal{O}(|\p_x^2\WW|+|\p_x\WW|)\nu_k(U)
+\mathcal{O}(|Z|+|Z_x|+|\WW|)
\nonumber
\\
&\equiv 
L(U)\TW-L(U)\Lambda+a(J(U)-J(V))\p_x^4\VV
\nonumber
\\
&\quad 
+
\sum_k\mathcal{O}(|\p_x^2\WW|+|\p_x\WW|)\nu_k(U).
\label{eq:3momo}
\end{align} 
Here, noting $\Lambda=\mathcal{O}(|Z|)$, 
it is immediate to see 
$$
L(U)\Lambda
\equiv 
a\p_x^2\{J(U)\p_x^2\Lambda\}
+2aA(U)(J(U)\p_x^3\Lambda, U_x).
$$ 
The second term of the right hand side 
is harmless in that 
$$
A(U)(J(U)\p_x^3\Lambda, U_x)
=\sum_k
\mathcal{O}(|\p_x^3\Lambda|)\nu_k(U)
=\sum_k\mathcal{O}(|\p_x\WW|)\nu_k(U).
$$ 
Hence we obtain 
\begin{align}
L(U)\Lambda
&\equiv 
a\p_x^2\{J(U)\p_x^2\Lambda\}
+
\sum_k\mathcal{O}(|\p_x\WW|)\nu_k(U)
\nonumber
\\
&=
-\frac{e_1}{2}
\p_x^2
\left\{
J(U)\p_x^2
\{
R(Z,U_x)U_x
\}
\right\}
\nonumber
\\
&\quad
+
\frac{e_2}{8}
\p_x^2
\{
J(U)\p_x^2
\{
R(J(U)U_x,U_x)J(U)Z
\}
\}
+
\sum_k\mathcal{O}(|\p_x\WW|)\nu_k(U).
\label{eq:4momo}
\end{align}
Combining 
\eqref{eq:2momo}, 
\eqref{eq:3momo}, and 
\eqref{eq:4momo}, 
we get 
\begin{align}
\p_t\TW
&\equiv
L(U)\TW+\frac{e_1}{2}I-\frac{e_2}{8}II
\nonumber
\\
&\quad +a(J(U)-J(V))\p_x^4\VV
+\sum_k\mathcal{O}(|\p_x^2\WW|+|\p_x\WW|)\nu_k(U), 
\label{eq:5momo}
\end{align}
where 
\begin{align}
I&:=
\p_x^2
\left\{
J(U)\p_x^2
\{
R(Z,U_x)U_x
\}
\right\}
-R(J(U)\p_x^2Z_{xx},U_x)U_x,
\label{eq:1comm}
\\
II&:=
\p_x^2
\{
J(U)\p_x^2
\{
R(J(U)U_x,U_x)J(U)Z
\}
\}
+R(J(U)U_x,U_x)\p_x^2Z_{xx}.
\label{eq:2comm}
\end{align}
\begin{remark} 
\label{remark:comm}
The structure of  $I$ and $II$ plays the essential role in our proof. 
Formally, we can write 
$\Lambda=(\Psi_1+\Psi_2)Z_{xx}$, 
where 
$$
\Psi_1=
-\frac{e_1}{2a}R(\p_x^{-2}\cdot,U_x)U_x, 
\quad
\Psi_2=
\frac{e_2}{8a}R(J(U)U_x,U_x)J(U)\p_x^{-2}.
$$
Using the formulation, we can observe that 
$$
-\frac{e_1}{2}I=
\left[
a\,\p_x^2\{J(U)\p_x^2\cdot\}, 
\Psi_1
\right]Z_{xx}, 
\quad
\frac{e_2}{8}I=
\left[
a\,\p_x^2\{J(U)\p_x^2\cdot\}, 
\Psi_2
\right]Z_{xx},
$$
where the bracket $\left[\cdot,\cdot\right]$ 
denotes the commutator. 
\end{remark}
We demonstrate the computation of $I$ and $II$, 
where we repeatedly use the formula
\eqref{eq:R7}, \eqref{eq:RP}. 
First, concerning $I$, a simple computation gives 
\begin{align}
&\p_x\{
R(Z,U_x)U_x
\}
\nonumber
\\
&=
R(\p_x\{P(U)Z\},U_x)U_x
+
R(Z,U_{xx})U_x
+
R(Z,U_x)U_{xx}
-A(R(Z,U_x)U_x,U_x),
\nonumber
\\
&\p_x^2\{
R(Z,U_x)U_x
\}
\nonumber
\\
&
=
R(\p_x^2\{P(U)Z\},U_x)U_x
+
2R(\p_x\{P(U)Z\},U_{xx})U_x
\nonumber
\\
&\quad
+
2R(\p_x\{P(U)Z\},U_x)U_{xx}
-2A(U)(R(\p_x\{P(U)Z\},U_x)U_x,U_x)
+\mathcal{O}(|Z|).
\nonumber
\end{align}
Using 
$J(U)A(U)(\cdot,\cdot)=0$, 
we obtain
\begin{align}
&J(U)\p_x^2\{
R(Z,U_x)U_x
\}
\nonumber
\\
&
=
J(U)R(\p_x^2\{P(U)Z\},U_x)U_x
+
2J(U)R(\p_x\{P(U)Z\},U_{xx})U_x
\nonumber
\\
&\quad
+
2J(U)R(\p_x\{P(U)Z\},U_x)U_{xx}
+\mathcal{O}(|Z|).
\nonumber
\end{align}
From this and $J(U)A(U)(\cdot,\cdot)=0$,  
we continue the computation and obtain 
\begin{align}
&\p_x^2\{
J(U)\p_x^2\{
R(Z,U_x)U_x
\}
\}
\nonumber
\\
&=
J(U)R(\p_x^4\{P(U)Z\},U_x)U_x
+
2\p_x(J(U))R(\p_x^3\{P(U)Z\},U_x)U_x
\nonumber
\\
&\quad
+
4J(U)R(\p_x^3\{P(U)Z\},U_{xx})U_x
+
4J(U)R(\p_x^3\{P(U)Z\},U_x)U_{xx}
\nonumber
\\
&\quad
+\mathcal{O}(|Z|+|Z_x|+|Z_{xx}|).
\label{eq:230}
\end{align}
Here, 
we see
\begin{align}
\p_x^3\{P(U)Z\}
&=P(U)\p_xZ_{xx}+\mathcal{O}(|Z|+|Z_x|+|Z_{xx}|),
\nonumber
\\
\p_x^4\{P(U)Z\}
&=P(U)\p_x^2Z_{xx}+4\p_x(P(U))\p_xZ_{xx}
+\mathcal{O}(|Z|+|Z_x|+|Z_{xx}|)
\nonumber
\\
&=
P(U)\p_x^2Z_{xx}-4A(U)(\p_xZ_{xx},U_x)
+\mathcal{O}(|Z|+|Z_x|+|Z_{xx}|).
\nonumber
\end{align}
Thus we have 
$P(U)\p_x^3\{P(U)Z\}\equiv P(U)\p_xZ_{xx}$ and 
$P(U)\p_x^4\{P(U)Z\}\equiv P(U)\p_x^2Z_{xx}$. 
In addition, from \eqref{eq:130} and $P(U)A(U)(\cdot,\cdot)=0$, 
we see 
$P(U)\p_x^3\{P(U)Z\}\equiv P(U)\p_x\WW$ and 
$P(U)\p_x^4\{P(U)Z\}\equiv P(U)\p_x^2\WW$. 
Substituting them into \eqref{eq:230}, we deduce
\begin{align}
I&\equiv
J(U)R(\p_x^2\WW,U_x)U_x
-R(J(U)\p_x^2\WW,U_x)U_x
+
2\p_x(J(U))R(\p_x\WW,U_x)U_x
\nonumber
\\
&\quad
+
4J(U)R(\p_x\WW,U_{xx})U_x
+
4J(U)R(\p_x\WW,U_x)U_{xx}.
\nonumber
\end{align} 
The computation below is now standard.
In the same way as above, using \eqref{eq:R1}, \eqref{eq:R3}, 
\eqref{eq:R4}, \eqref{eq:R6},  
we deduce  
\begin{align}
&J(U)R(\p_x^2\WW,U_x)U_x-R(J(U)\p_x^2\WW,U_x)U_x
\nonumber
\\
&=R(\p_x^2\WW,U_x)J(U)U_x+R(\p_x^2\WW,J(U)U_x)U_x
\nonumber
\\
&=-R(U_x,J(U)U_x)\p_x^2\WW
-R(J(U)U_x,\p_x^2\WW)U_x
+R(\p_x^2\WW,J(U)U_x)U_x
\nonumber
\\
&=
2R(\p_x^2\WW,J(U)U_x)U_x
+R(J(U)U_x,U_x)\p_x^2\WW
\nonumber
\\
&\equiv 
2R(\p_x^2\WW,J(U)U_x)U_x
+\p_x\{R(J(U)U_x,U_x)\p_x\WW\}
\nonumber
\\
&\quad 
-2R(J(U)\UU,U_x)\p_x\WW
+\sum_k\mathcal{O}(|\p_x\WW|)\nu_k(U).
\label{eq:430}
\end{align} 
Using \eqref{eq:kaehler2} and $R(\p_x\WW,U_x)U_x\perp \nu_k(U)$, 
we see 
\begin{align}
2\p_x(J(U))R(\p_x\WW,U_x)U_x
&=-2A(U)(J(U)R(\p_x\WW,U_x)U_x
\nonumber
\\
&
=\sum_k\mathcal{O}(|\p_x\WW|)\nu_k(U).
\label{eq:530}
\end{align}
Using \eqref{eq:ppp1}-\eqref{eq:ppp2}, we see
\begin{align}
&4J(U)R(\p_x\WW,U_{xx})U_x
+
4J(U)R(\p_x\WW,U_x)U_{xx}
\nonumber
\\
&=
4R(\p_x\WW,\UU)J(U)U_x
+
4R(\p_x\WW,U_x)J(U)\UU
\nonumber
\\
&=
4R(J(U)\UU,U_x)\p_x\WW
+2(B_2(U)+B_3(U))\p_x\WW. 
\label{eq:630}
\end{align}
Combining \eqref{eq:430}-\eqref{eq:630}, we obtain 
\begin{align}
I&\equiv 
2R(\p_x^2\WW,J(U)U_x)U_x
+\p_x\{R(J(U)U_x,U_x)\p_x\WW\}
+2R(J(U)\UU,U_x)\p_x\WW
\nonumber
\\
&\quad 
+2(B_2(U)+B_3(U))\p_x\WW
+\sum_k\mathcal{O}(|\p_x\WW|)\nu_k(U).
\label{eq:I_r}
\end{align}
Second, we compute $II$. 
The way of the computation of $II$ is now almost 
same as that for $I$.  
The result of the computation is as follows: 
\begin{align}
&\p_x\{
R(J(U)U_x,U_x)J(U)Z
\}
=\p_x\{
J(U)R(J(U)U_x,U_x)Z
\}
\nonumber
\\
&=
J(U)R(J(U)U_x,U_x)\p_x\{P(U)Z\}
+
\p_x(J(U))R(J(U)U_x,U_x)Z
\nonumber
\\
&\quad
+J(U)R(\p_x\{J(U)U_x\},U_x)Z
+J(U)R(J(U)U_x,U_{xx})Z.
\nonumber
\end{align}
\begin{align}
&\p_x^2\{
R(J(U)U_x,U_x)J(U)Z
\}
\nonumber
\\
&=
J(U)R(J(U)U_x,U_x)\p_x^2\{P(U)Z\}
+
2\p_x(J(U))R(J(U)U_x,U_x)\p_x\{P(U)Z\}
\nonumber
\\
&\quad
+2J(U)R(\p_x\{J(U)U_x\},U_x)\p_x\{P(U)Z\}
\nonumber
\\
&\quad
+2J(U)R(J(U)U_x,U_{xx})\p_x\{P(U)Z\}
+\mathcal{O}(|Z|).
\nonumber
\end{align}
\begin{align}
&J(U)\p_x^2\{
R(J(U)U_x,U_x)J(U)Z
\}
\nonumber
\\
&=
-R(J(U)U_x,U_x)\p_x^2\{P(U)Z\}
-2R(\p_x\{J(U)U_x\},U_x)\p_x\{P(U)Z\}
\nonumber
\\
&\quad
-2R(J(U)U_x,U_{xx})\p_x\{P(U)Z\}
+\mathcal{O}(|Z|).
\nonumber
\end{align}
\begin{align}
&\p_x^2\{
J(U)\p_x^2\{
R(J(U)U_x,U_x)J(U)Z
\}
\}
\nonumber
\\
&=
-R(J(U)U_x,U_x)\p_x^4\{P(U)Z\}
-4R(\p_x\{J(U)U_x\},U_x)\p_x^3\{P(U)Z\}
\nonumber
\\
&\quad
-4R(J(U)U_x,U_{xx})\p_x^3\{P(U)Z\}
\nonumber
\\&\quad
+\sum_k\mathcal{O}(|\p_x\WW|)\nu_k(U)
+\mathcal{O}(|Z|+|Z_x|+|Z_{xx}|)
\nonumber
\\
&\equiv 
-R(J(U)U_x,U_x)\p_x^2\WW
-4R(J(U)\UU,U_x)\p_x\WW
\nonumber
\\
&\quad
-4R(J(U)U_x,\UU)\p_x\WW
+\sum_k\mathcal{O}(|\p_x\WW|)\nu_k(U)
\nonumber
\\
&=
-R(J(U)U_x,U_x)\p_x^2\WW
-8R(J(U)\UU,U_x)\p_x\WW
+\sum_k\mathcal{O}(|\p_x\WW|)\nu_k(U).
\label{eq:830}
\end{align}
From \eqref{eq:830} and 
$R(J(U)U_x,U_x)\p_x^2Z_{xx}\equiv R(J(U)U_x,U_x)\p_x^2\WW$, 
we obtain 
\begin{align}
II&\equiv 
-8R(J(U)\UU,U_x)\p_x\WW
+\sum_k\mathcal{O}(|\p_x\WW|)\nu_k(U).
\label{eq:II_r}
\end{align}
Substituting \eqref{eq:I_r} and \eqref{eq:II_r} into 
\eqref{eq:5momo}, and using 
$\p_x^2\TW\equiv \p_x^2\WW$ and 
$\p_x\TW\equiv \p_x\WW$, 
we get 
\begin{align}
\p_t\TW
&\equiv 
L(U)\TW
+e_1R(\p_x^2\TW,J(U)U_x)U_x
+\frac{e_1}{2}\p_x\{
R(J(U)U_x,U_x)\p_x\TW
\}
\nonumber
\\&\quad 
+(e_1+e_2)R(J(U)\UU,U_x)\p_x\TW
+e_1(B_2(U)+B_3(U))\p_x\TW
\nonumber
\\
&\quad
+a(J(U)-J(V))\p_x^4\VV
+\sum_k\mathcal{O}(|\p_x^2\TW|+|\p_x\TW|)\nu_k(U).
\label{eq:momomo}
\end{align}
We now arrived at the estimate for
$\lr{\p_t\TW,\TW}$.  
Since $\TW=\WW+\mathcal{O}(|Z|)$, 
$\lr{L(U)\TW,\TW}$ has the same type of the estimate 
as that of $\lr{L(U)\WW,\WW}$ which is bounded by 
the right hand side of \eqref{eq:momo}. More precisely, we deduce  
\begin{align}
\lr{L(U)\TW,\TW}
&\leqslant 
(-a+b)\lr{
R(\p_x^2\TW,J(U)U_x)U_x,\TW
}
\nonumber
\\
&\quad
+\left(
\frac{a}{2}-\frac{b}{2}+3c
\right)\lr{
R(J(U)\UU,U_x)\p_x\TW,\TW
}
+C\,\widetilde{D}^2(t)
\label{eq:2momomo}
\end{align}
for almost every $t$.
The other terms (except the second and the fourth terms) 
of the right hand side of \eqref{eq:momomo} can be handled 
in the same way 
as we obtain 
\eqref{eq:10mm}, 
\eqref{eq:11mm},
\eqref{eq:423add},
and
\eqref{eq:18e}. 
Having them in mind, we take $e_1$ and $e_2$ so that 
$$
e_1=a-b,
\quad
e_2=-e_1-\frac{a}{2}+\frac{b}{2}-3c
=
-\frac{3a}{2}
+\frac{3b}{2}-3c.
$$
Then, we derive the following estimate 
\begin{align}
&\frac{1}{2}\frac{d}{dt}
\|\TW\|_{L^2}^2
=\lr{\p_t\TW,\TW}
\nonumber
\\
&\leqslant
(e_1-a+b)\lr{
 R(\p_x^2\TW,J(U)U_x)U_x,\TW
 }
 \nonumber
 \\
 &\quad
 +\left(
 e_1+e_2+
 \frac{a}{2}-\frac{b}{2}+3c
 \right)\lr{
 R(J(U)\UU,U_x)\p_x\TW,\TW
 }
 +C\,\widetilde{D}^2(t)
 \nonumber
 \\
 &=
 C\,\widetilde{D}^2(t)
\label{eq:mmk}
\end{align}
for almost every $t\in (-T,T)$. 
This combined with \eqref{eq:notmainineq} 
shows the desired estimate \eqref{eq:De}, 
and thus we complete the proof. 
\end{proof}
\section*{Acknowledgments}
The author has been supported by 
JSPS Grant-in-Aid for Scientific Research (C) 
Grant Number JP20K03703.

\end{document}